\documentclass[a4paper]{amsart}

\usepackage{amsmath, hyperref}
\usepackage{amssymb}
\usepackage{amsfonts}

\usepackage{amsaddr} 
\usepackage[latin1]{inputenc}
\usepackage{latexsym, 
}

\usepackage{graphicx}
\usepackage{tikz}
\usepackage{bussproofs}

\usepackage{enumerate}  
\usepackage{mathtools}  
\usepackage{bm}         
\usepackage{url}        

\usepackage{tikz}

\usetikzlibrary[arrows,snakes]
\tikzstyle{bull}=[circle,draw=black,fill=black!80]
\tikzstyle{holl}=[circle,draw=black]

\usepackage{microtype}
\makeatletter
\def\MT@register@subst@font{\MT@exp@one@n\MT@in@clist\font@name\MT@font@list
   \ifMT@inlist@\else\xdef\MT@font@list{\MT@font@list\font@name,}\fi}
\makeatother

\makeatletter
\def\@listi{\leftmargin\leftmargini
              \parsep 0\p@ \@plus2\p@
              \topsep 2\p@ \@plus1\p@
              \itemsep 2\p@ \@plus1\p@}
\let\@listI\@listi
\@listi
\makeatother

\newtheorem{theorem}{Theorem}[section]
\newtheorem{lemma}[theorem]{Lemma}
\newtheorem{proposition}[theorem]{Proposition}

\newtheorem{corollary}[theorem]{Corollary}

\theoremstyle{definition}
\newtheorem{definition}[theorem]{Definition}

\newtheorem{example}[theorem]{Example}

\newtheorem{remark}[theorem]{Remark}

\newcommand{\la}{\langle}
\newcommand{\ra}{\rangle}

\newcommand{\Al}[1][A]{\ensuremath{\mathbf{#1}} }

\newcommand{\B}{\mathcal{B}}

\renewcommand{\phi}{\varphi}

\newcommand{\A}{\Al}

 \newcommand{\Fm}{\Al[Fm]}

 \newcommand{\K}{\mathsf{K}}

 \newcommand{\V}{\mathsf{V}}

 \newcommand{\Matr}{\mathsf{Matr}}

 \newcommand{\Alg}{\mathsf{Alg}}

 \newcommand{\imp}{\mathbin\to}

  \newcommand{\ISL}{\mathcal{IS}_{\leq}}

 \let\Omega=\varOmega

 \let\Gamma=\varGamma

 \let\Lambda=\varLambda

 \let\Sigma=\varSigma

 \let\Psi=\varPsi

 \let\Delta=\varDelta

 \let\Pi=\varPi

 \let\Theta=\varTheta

\newcommand{\Leibniz}{{\boldsymbol{\varOmega}}}
 \DeclareMathOperator{\Log}{Log}

\newcommand{\CMt}{\mathcal{M}}

\newcommand{\tuple}[1] {{\langle #1\rangle}}
\let \ders \vdash  
\let \Lo \vdash
\let \derm \vartriangleright

\newcommand{\sub}  {{\mathsf{sub}}}

\newcommand{\ISA}{\mathsf{IS}}
  
 \newcommand{\HO}{\mathbb{H}}
 \newcommand{\PR}{\mathbb{P}}
 \newcommand{\SU}{\mathbb{S}}
  \newcommand{\II}{\mathbb{I}}
  \newcommand{\VV}{\mathbb{V}}
  \newcommand{\QQ}{\mathbb{Q}}

\newcommand{\Mt}{\mathbb{M}}

\newcommand{\nnot}{\mathop{\sim}}                  

\let\imp\Rightarrow                                
\newcommand{\alg}[1]{\mathbf{#1}}                  

\let\epsilon=\Phi

\newcommand{\eval}[2][\right]{\relax
   \ifx#1\right\relax \left.\fi#2#1\rvert}

\newcommand{\ttop} {\hat{1}}
\newcommand{\bbot} {\hat{0}}

\newcommand{\Ax}             {\mathsf{Ax}}
\newcommand{\R}             {\mathsf{R}}
\newcommand{\Val}{\textrm{Val}}

\begin{document}

\title{Logics of involutive Stone algebras}

\author{S\'ergio Marcelino}
\address{SQIG - Instituto de Telecomunica\c{c}\~oes, Portugal\\
Departamento de Matem\'atica, Instituto Superior T\'ecnico, Universidade de Lisboa
}
\email{smarcel@math.tecnico.ulisboa.pt}
%

\author{ Umberto Rivieccio$^*$}
\address{Departamento de Inform\'atica e Matem\'atica Aplicada,
Universidade Federal do Rio Grande do Norte,
Natal (RN), Brasil}
\email{urivieccio@dimap.ufrn.br, phone: +55 84 3215 3814}
%

\date{\today}

\maketitle

\begin{abstract}
An involutive Stone algebra (IS-algebra) is a  structure that is simultaneously 
a De Morgan algebra and a Stone algebra (i.e.~a pseudo-complemented distributive lattice
satisfying the well-known Stone identity, $\nnot x \lor \nnot \nnot x \approx 1$).
IS-algebras have been studied algebraically and topologically since the 1980's, but a corresponding 
logic  (here denoted $\ISL$) has been  introduced only very recently. 
The logic $\ISL$ is the departing point for the present study, which we then extend to a wide family of previously unknown
logics
defined from IS-algebras.
We show that $\ISL$ is a conservative expansion of the Belnap-Dunn four-valued
logic (i.e.~the order-preserving logic of the variety of De Morgan algebras), and we give a finite 
Hilbert-style axiomatization for it. 
More generally, we introduce a 
method for expanding conservatively
every super-Belnap logic 
so as to obtain an extension of  $\ISL$.
We show that every logic thus defined can be axiomatized by adding a fixed
finite set of rule schemata to the corresponding super-Belnap base logic.
We also consider a few sample extensions of $\ISL$ that cannot be obtained in the above-described way,
but can nevertheless be axiomatized finitely by other methods.
Most of our axiomatization results are obtained in two steps:  through a multiple-conclusion
calculus first, which we then reduce to a traditional one. 
The multiple-conclusion axiomatizations introduced
in this process, being analytic, are  of independent interest from a proof-theoretic standpoint.
Our results entail that the lattice of super-Belnap logics (which is known to be uncountable)
embeds into the lattice of extensions of $\ISL$. Indeed, as in the super-Belnap case, we establish 
that  the \emph{finitary} extensions of $\ISL$ are already  uncountably many. 
\end{abstract}

\section{Introduction}
\label{sec:intro}

Involutive Stone algebras (from now on, IS-algebras) were first considered in the papers~\cite{CiGa81,CiGa83} 
within a study of finite-valued \L ukasiewicz logics and,
more specifically, in connection with the algebraic structures nowadays known as
\L ukasiewicz-Moisil algebras. The term `involutive' is due to the observation that every
IS-algebra 
has  a primitive negation operation $\nnot$ that satisfies the involutive law ($\nnot \nnot x \approx x$),
whereas 
`Stone' refers to the existence of a term-definable
pseudo-complement
 operation $\neg$ that 
satisfies the
well-known Stone identity $\neg x \lor \neg \neg x \approx 1 $.
From an algebraic point of view, 
IS-algebras
are a variety of 
De Morgan algebras endowed with an additional unary operation
(here denoted by $\nabla$);
alternatively,
 IS-algebras can  be viewed as the subclass of  De Morgan algebras that
satisfy certain structural properties  ensuring the definability of $\nabla$.
The structural   connection between De Morgan  
and IS-algebras 
 will indeed play a prominent role in the present paper.

De Morgan algebras form a variety  that is well known
in non-classical logic as the algebraic counterpart of~$\B$, the four-valued Belnap-Dunn logic~\cite{Be77,F97}.
From a logical point of view, $\B$ can be viewed as a weakening of classical two-valued logic designed to allow for both \emph{paraconsistency}
(in that $\mathcal{B}$ does not validate the rule $p \land \nnot p \vdash q$, known as  \emph{ex contradictione quodlibet})
and \emph{paracompleteness} (in that $\mathcal{B}$ does not validate the principle of excluded middle, $\vdash p \lor \nnot p$). 
De Morgan algebras 
can thus be viewed as a generalization of a Boolean algebras 
on which the operation $\nnot$ (that interprets the negation connective) need not be a Boolean complement,
i.e.~
the classical laws $x \land \nnot x \approx 0$  and $x \lor \nnot x \approx 1$ need not be satisfied.
The involutive  and the De Morgan laws   are however
valid on every De Morgan algebra (see Definition~\ref{def:dm}). 

As mentioned above,  involutive Stone algebras may be regarded as a  subclass of De Morgan algebras
characterized by certain structural properties. From this perspective,
an involutive Stone algebra may be viewed as a De Morgan algebra $\Al$ having an additional unary
operation $\nabla$ that receives an arbitrary element $a \in A$ as argument
and outputs a certain `classical' element $\nabla a$ (i.e.~an element that possesses a Boolean complement in $\Al$).
Just as not every distributive lattice can be equipped with a Boolean complement operation, so 
not every De Morgan algebra 
can be endowed with an operation $\nabla$  meeting the above requirement.
However, if such an operation 
is definable, 
then it is unique. 

One might say that,
on every
De Morgan algebra $\Al$,
the behaviour of $\nabla$ 
provides 
a measure of how far  $\Al$ is from being Boolean:
the limit cases being, at  one end of the spectrum, Boolean algebras themselves
 (on which $\nabla$ is the identity map) and, at the other, the algebras (such as those depicted in Figure~\ref{fig:1}) where the 
 only Boolean elements are the top and the bottom.
These structural requirements on  $\nabla$ can be completely captured by means of identities~\cite[Thm.~2.1]{CiGa83}.
Therefore, regardless of the preceding considerations, IS-algebras can be simply introduced as a variety of
De Morgan algebras
having an extra unary operation $\nabla$ that is required to satisfy four additional identities
(see Definition~\ref{def:isa}).

A logic associated to IS-algebras has been considered for the first time in~\cite{Ca-MSc,CaFi18}. 
In the present paper, we shall denote  this logic by $\ISL$, suggesting 
that
$\ISL$ is the order-preserving logic canonically associated to the variety of  IS-algebras (see Section~\ref{sec:prel} for the relevant definitions). As we will show, 
$\ISL$ is a conservative expansion of the Belnap-Dunn logic $\B$, 
which is itself the order-preserving logic of the variety of De Morgan lattices.
We are moreover going to prove that,
between the logics extending $\B$ (known as \emph{super-Belnap logics} since the paper~\cite{Ri12b}) and 
the extensions of $\ISL$,  a connection can be established and exploited
in order to obtain a number of non-trivial results.

The background facts we shall need on the Belnap-Dunn logic can be found in~\cite{F97}, including
a complete Hilbert-style axiomatization and a characterization of the reduced matrix models (see  Section~\ref{sec:prel}).
For further information on super-Belnap logics, we refer the reader to the papers~\cite{Ri12b,AlPrRi16,Adam}, from which
we shall also import a few results as needed. 

The rest of the paper is organized as follows.
Section~\ref{sec:prel} introduces the generic algebraic and logical notions that will be used in the following ones.
 Section~\ref{sec:isal} contains the basic algebraic results on De Morgan and involutive Stone algebras.
In Section~\ref{sec:sem} we look at the logic $\ISL$ of involutive Stone algebras from a semantical point of view. 
We observe that  $\ISL$ non-protoalgebraic (Proposition~\ref{prop:nonprot}) and
can be characterized by a single finite matrix (Proposition~\ref{prop:onemat}).
We further introduce a simple operation on logical matrices that allows us to associate, 
to any given super-Belnap logic, a logic
extending
$\ISL$ is such a way that the latter is a conservative expansion of the former
(Lemma~\ref{lem:consex}). This entails that the lattice of super-Belnap logics
is embeddable into the lattice of extensions of $\ISL$ (Corollary~\ref{cor:latlog}),
which in turn tells us that the latter must have at least the cardinality of the continuum (Corollary~\ref{cor:cardlat}).

In Section~\ref{sec:ax} we present a uniform method of axiomatizing all the extensions of
$\ISL$ that are defined from super-Belnap logics via the construction introduced in Section~\ref{sec:sem}
(Corollary~\ref{cor:wck}, Proposition~\ref{prop:axnablasuperbelnap});
a complete axiomatization for $\ISL$ is  thus obtained as a special application 
(Example~\ref{ex:logis}).
In order to achieve these results, we take a little detour through
the realm of multiple-conclusion logics and the calculi that correspond to them.
%
In Subsection~\ref{ss:addsb}
a number of logics extending $\ISL$ (corresponding to substructures of the
matrix that defines $\ISL$) are axiomatized by a uniform application of the general method;
these include logics obtained by adding the $\nabla$ connective
to well-known extensions of $\B$, such as G.~Priest's \emph{Logic of Paradox} 
and the strong and weak three-valued logics due to S.~C.~Kleene.
In Subsection~\ref{ss:oth} we axiomatize a few extensions of $\ISL$ are not obtained
in this way 
from a super-Belnap logic, among which we find 
the three-valued \L ukasiewicz(-Moisil) logic.
For the latter results we 
cannot apply the above-mentioned method, so we need to take a longer detour through
multiple-conclusion logics and analytical calculi (Subsection~\ref{ss:anal}). 
Finally, Section~\ref{sec:conc} contains a few concluding remarks and suggestions
for future research.

\section{Algebraic and logical preliminaries}
\label{sec:prel}

In this Section we recall the main algebraic and logical notions that will be needed in the following ones.
We assume familiarity with basic results of lattice theory~\cite{DaPr90}, universal algebra~\cite{BuSa00} and the general theory of 
logical calculi~\cite{W88,FJa09,F16}. 

We shall denote by $\A, \alg{B}$ etc.~algebras over a given algebraic similarity type $\Sigma$.
The set of $\Sigma$-homomorphisms between two algebras $\A$  and $\alg{B}$
will be denoted by  $\mathsf{Hom}(\A,\mathbf{B})$.
Given $\Sigma$-algebras $\A, \alg{B}$ and a sub-signature $\Sigma' \subseteq \Sigma$,
we denote by $\mathsf{Hom}_{\Sigma'}(\A,\mathbf{B})$ the set of functions
$h \colon A \to B $ that are only required to preserve the operations in $\Sigma'$.
The algebra of formulas over a signature $\Sigma$ will be denoted by $\alg{Fm}_{\Sigma}$
(or simply by $\alg{Fm}$, if $\Sigma$ is clear from the context), and its elements by $\phi, \psi$ etc.
Given a class $\K$ of similar algebras, we denote by $\II (\K)$, $\HO (\K)$, $\SU (\K)$, $\PR (\K)$,
$\PR_S (\K)$
 the classes formed by closing
$\K$ under (respectively), isomorphisms, homomorphisms, subalgebras, direct products and subdirect products. 
A \emph{variety}  is a class $\K$ of algebras that is closed under $\HO, \SU, \PR$, or, equivalently,
that is definable by means of algebraic identities.
A \emph{quasivariety} is a class $\K$ of algebras that is definable by means of quasi-identities, that is, implications having a finite
number of identities as premiss and a single identity as conclusion. 
The variety (resp.~quasivariety) generated by 
$\K$ will be denoted by
$\VV(\K)$ and $\QQ(\K)$. Every variety $\V$ is generated by the class $\V_{si}$ of its
subdirectly irreducible members, defined as follows: an algebra $\A$ is \emph{subdirectly irreducible} 
if $\A$ has a minimum congruence above the identity relation (as a special case, we say that $\A$ is  \emph{simple} if $\A$ has
exactly two congruences).
For every variety $\V$, we have 
$\V = \VV ( \V_{si}) = \II \PR_S (\V_{si})$.

We view a  \emph{(propositional, single-conclusion) logic} as a structural consequence relation on $\wp (Fm) \times Fm$ (see e.g.~\cite[Def.~1.5]{F16}).
The symbol $\Lo$ 
will be used to denote arbitrary logics.

We say that a logic $\Lo_2$ is an \emph{extension} of $\Lo_1$ when both logics share the same propositional language $\Sigma$
and ${\Lo_1} \, \subseteq \, {\Lo_2}$. The family of all extensions of a logic $\Lo$ forms a complete lattice (in which the meet is the intersection); in this paper we shall be concerned with the lattice of extensions of the logic $\ISL$ of involutive Stone algebras, and will relate it to the lattice of extensions of the Belnap-Dunn logic $\mathcal{B}$ (i.e.~the super-Belnap logics). 
We say that a logic $\Lo_2$ over a language $\Sigma_2$ is an \emph{expansion} of a logic $\Lo_1$ over  $\Sigma_1$ when $\Sigma_1 \subseteq \Sigma_2$ and 
${\Lo_1} \, \subseteq \, {\Lo_2}$. We speak of a \emph{conservative expansion} when
both consequence relations coincide on the formulas over $\Sigma_1$. 

A \emph{(logical) matrix} is a pair
$\Mt = \la \A, D \ra$ where $\A$ is an algebra
and $D \subseteq A$ is a subset of \emph{designated elements}.
One  defines the notions of isomorphism, homomorphism, submatrix and product of matrices 
as straightforward extensions of the corresponding universal algebraic constructions (see~\cite{W88,FJa09} for details). 
Given a matrix $\Mt = \la \A, D \ra$ with  $\A$  a $\Sigma$-algebra, 
we let
$\Val(\Mt)=\mathsf{Hom}(\Fm_{\Sigma},\mathbf{A})$.
 We denote by $\Log$ the mapping that associates a logic to a class of matrices in the standard fashion.
Indeed, each matrix determines a logic 
(denoted $\Log{\Mt}$ or $\vdash_{\Mt}$) as follows: for all
$\Gamma \cup \{ \phi \} \subseteq Fm$,
one lets
$\Gamma \vdash_{\Mt} \phi $
iff,
for every valuation $v \colon \Al[Fm] \to \A$,
$v[\Gamma ] = \{ v(\gamma) : \gamma \in \Gamma \} \subseteq D$ entails $v(\phi) \in D$. 
To a class of matrices
$\CMt = \{ \Mt_i : i \in I \}$,
we associate the logic
$\Log {\CMt}= {\ders_\CMt} : = \bigcap \{\ders_{\Mt_i} : i \in I \} $.
We say that a matrix $\Mt$ is a \emph{model} of a logic $\ders$
when 
${\Lo} \subseteq {\ders_\Mt}$,
that is,
when $\Gamma \vdash  \phi$ entails $\Gamma \vdash_{\Mt} \phi $, for all 
$\Gamma \cup \{ \phi \} \subseteq Fm$.

Every matrix $\Mt = \la \A, D \ra$
has an associated 
 \emph{Leibniz congruence} $\Leibniz_{ \alg{A}} ( D) $,
 which is
 the greatest congruence on $\A$ that is compatible with $D$ in the following sense:
 for all $a \in D$ and $b \in A$, if $\la a,b \ra \in \Leibniz_{ \alg{A}} ( D)$,
 then $b \in D$. This property allows one to define the quotient matrix
 $\Mt^* = \la \A / {\Leibniz_{ \alg{A}} ( D)} , D  / {\Leibniz_{ \alg{A}} ( D)} \ra$,
 which is known as the \emph{reduction} of  $\Mt$.
 A matrix $\Mt$ is \emph{reduced} if $\Leibniz_{ \alg{A}} ( D)$ is the identity relation,
 so no further reduction is possible. Reduced matrices are important in the study of algebraic models
 of logics, because, for every matrix $\Mt$, one has 
 $\Log \Mt = \Log \Mt^*$.
 It follows that every logic coincides with the logic determined by the class of all its reduced matrix models. 
 In algebraic logic, two classes of algebras, $\Alg^* (\Lo)$ and $\Alg (\Lo)$,
are traditionally associated to a given logic $\Lo$. 
The former is defined as follows:
$
\Alg^* (\Lo) := \{ \A :  \text{ there is } D \subseteq A \text{ such that } \la  \A, D \ra \text{ is a reduced matrix for } \Lo\}
$
By the characterization of~\cite[Thm.~2.23]{FJa09}, 
the  latter class can be introduced as follows $\Alg (\Lo) := \PR_S (\Alg^* (\Lo) )$.

Let $\K$ be a class of algebras such that each $\A \in \K$ has a semilattice
reduct with top element $1$.
From $\K$ one  can obtain a finitary logic $ \vdash^{\leq}_{\K} $ as follows.
One lets $\emptyset \vdash^{\leq}_{\K}  \phi$ if and only if
the identity $\phi \approx 1$ is valid in $\K$ and,
for all $\Gamma \cup \{ \phi \} \subseteq Fm$ such that $\Gamma \neq \emptyset$, 
one lets
$
\Gamma \vdash^{\leq}_{\K} \phi
$
iff
there are 
$\gamma_1, \ldots, \gamma_n \in \Gamma$
such that 
the identity 
$  \gamma_1 \land \ldots \land \gamma_n \land \phi \approx \phi $
is valid in $\K$. 
We shall call $\vdash^{\leq}_{\K} $
the \emph{order-preserving logic} associated to $\K$.
Observe that, by definition, $\K$ and $\VV(\K)$ define the same logic;
thus, the order-preserving logics considered in the literature are usually associated
 to varieties
of semilattice-based algebras.
We note that, if $\V$ is a variety of algebras having a bounded distributive lattice reduct (as will always be the case in the present
paper), then $\vdash^{\leq}_{\V} $ coincides with the logic defined by the class of matrices
$\{ \la \A, F \ra  :  \A \in \V, \, F \subseteq A \text{ is a (non-empty) lattice filter of } \A 
\}$.

%
%
%


\section{De Morgan and involutive Stone algebras}
\label{sec:isal}

In this Section we recall the main
definitions and basic results on the classes of algebras involved. 

\begin{definition}
\label{def:dm}
A \emph{
De Morgan lattice} is an algebra 
$\Al = \la A; \land, \lor, \nnot \ra$ of type $\la 2,2,1 \ra$
such that $\la A; \land, \lor \ra$ is a distributive lattice and  the following identities are satisfied: 
\begin{enumerate}[(DM1)]
\item \label{Itm:D1} $\nnot (x \lor y) \approx  \nnot x \land \nnot y$. 
\item \label{Itm:D2} $ \nnot (x \land y) \approx  \nnot  x \lor  \nnot y$. 
\item \label{Itm:D3} $ x \approx \nnot  \nnot x$. 
\end{enumerate}
A  \emph{
De Morgan algebra} is a 
De Morgan lattice whose lattice reduct is bounded (thus we include constant symbols $\bot$ and $\top$ in the algebraic signature) and satisfies
the following identities: 
$\nnot \top \approx \bot$ and $\nnot \bot \approx \top$.
\end{definition}

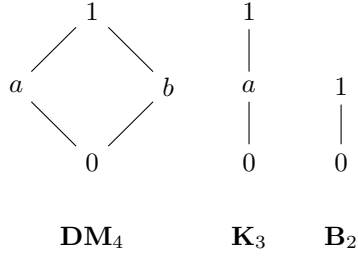
\begin{figure}[h]
\caption{(All the)  subdirectly irreducible De Morgan algebras.}
\label{fig:2}

\bigskip

\begin{center}
\begin{tikzpicture}[scale=1,
  dot/.style={circle,fill,inner sep=1.6pt,outer sep=2pt}]
  \node (K6b) at (0,0) {$0 $};
  \node (K6d) at (1,1) {$b$};
    \node (K6g) at (0,2) {$1$};
            \node (K6i) at (-1,1) {$a$};
    \draw[-] (K6d) -- (K6g);
    \draw[-] 
    (K6b) -- (K6d);
    \draw[-] (K6b) -- (K6i) -- (K6g);
  \node at (0,-1) {$\alg{DM}_4$};
\end{tikzpicture}
\hspace{8pt}
\begin{tikzpicture}[scale=1,
  dot/.style={circle,fill,inner sep=1.6pt,outer sep=2pt}]
  \node (K6b) at (1,0) {$0 $};
    \node (K6g) at (1,2) {$1$};
            \node (K6i) at (1,1) {$a$};
    \draw[-] (K6b) -- (K6i) -- (K6g);
  \node at (1,-1) {$\alg{K}_3$};
\end{tikzpicture}
\hspace{8pt}
\begin{tikzpicture}[scale=1,
  dot/.style={circle,fill,inner sep=1.6pt,outer sep=2pt}]
  \node (K6b) at (1,0) {$0 $};
        \node (K6h) at (1,1) {$1$};
  \draw[-] (K6b) -- (K6h);
    ;
  \node at (1,-1) {$\alg{B}_2$};
\end{tikzpicture}
\end{center}
\end{figure}

Figure~\ref{fig:2} depicts the (only) subdirectly irreducible De Morgan algebras.
On each algebra, the lattice operations are determined by the diagram.
The negation is defined on $\alg{DM}_4$ by $\nnot 0 = 1$, $\nnot 1 = 0$,
$\nnot a = a$ and $\nnot b = b$. These prescriptions apply to 
$\alg{K}_3$ and $\alg{B}_2$ as well viewed as subalgebras of $\alg{DM}_4$.
Obviously $\alg{B}_2$ is the two-element Boolean algebra, and 
$\alg{K}_3$ is the three-element  Kleene algebra
associated to the three-valued logics originating from the work of S.C.~Kleene\footnote{Formally,
a \emph{Kleene lattice (algebra)} is defined as a De Morgan lattice (algebra)
that satisfies $x \land \nnot x \leq y \lor \nnot y$. It is well known that the 
variety of Kleene lattices (algebras)
is 
$\VV (\alg{K_3})$.}. 

\begin{definition}
\label{def:isa}
An \emph{
involutive Stone algebra} (IS-algebra) is an algebra 
$\Al = \la A; \land, \lor, \nnot, \nabla, 0, 1 \ra$ of type $\la 2,2,1,1,0,0 \ra$
such that $ \la A; \land, \lor, \nnot, \bot, \top \ra$ is a De Morgan algebra and
 the following identities are satisfied: 
\begin{enumerate}[({I}S1)]
\item \label{Itm:IS1} $\nabla \bot \approx  \bot$. 
\item \label{Itm:IS2} $ x   \approx  x \land \nabla x$. 
\item \label{Itm:IS3} $\nabla( x \land y) \approx  \nabla x \land \nabla y $. 
\item \label{Itm:IS4} $\nnot \nabla x \land \nabla x \approx  \bot $. 
\end{enumerate}
\end{definition}

The class of IS-algebras will be denoted $\ISA$.
The name `involutive Stone algebras' is motivated by the following observation. For every
IS-algebra
 $\Al = \la A; \land, \lor, \nnot, \nabla, \bot, \top \ra$, 
 the operation $\neg$ that realizes  the term $\neg x := \nnot \nabla x$ 
  is a pseudo-complement;
moreover, $\Al$ satisfies the so-called Stone identity $\neg x \lor \neg \neg x \approx \top$. Hence,
$\la A; \land, \lor, \neg, \bot, \top \ra$ is a Stone algebra\footnote{Formally, a \emph{Stone algebra} can be defined as
a bounded distributive lattice $\la A; \land, \lor, \neg, \bot, \top \ra$ endowed by an extra unary operation $\neg$
that satisfies, for all $a,b \in A$, the following requirements: (i) $a \land b = 0$ iff $a \leq \neg b$, and (ii) $\neg a \lor \neg \neg a = \top$.}.
Conversely, given an algebra $\la A; \land, \lor, \nnot, \neg, \bot, \top \ra$  that has both a De Morgan negation and a pseudo-complement
operation, upon defining $\nabla x : = \neg \neg x $, 
one has that $\la A; \land, \lor, \nnot, \nabla, 0, 1 \ra$ is an IS-algebra if and only if the following
identity is satisfied:
$\neg x = \nnot \neg \neg x $
~\cite[Remark~2.2]{CiGa83}.

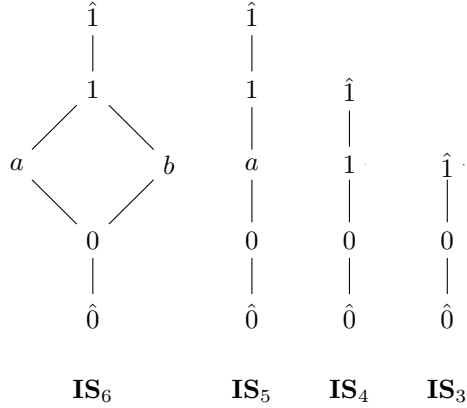
\begin{figure}[h]
\caption{(All the)  subdirectly irreducible IS-algebras.}
\label{fig:1}

\bigskip

\begin{center}

\begin{tikzpicture}[scale=1,
  dot/.style={circle,fill,inner sep=1.6pt,outer sep=2pt}]
  \node (K6a) at (0,-1) {$\bbot$};
  \node (K6b) at (0,0) {$0 $};
  \node (K6d) at (1,1) {$b$};
    \node (K6g) at (0,2) {$1$};
        \node (K6h) at (0,3) {$\ttop$};
            \node (K6i) at (-1,1) {$a$};
  \draw[-] (K6g) -- (K6h);
    \draw[-] (K6d) -- (K6g);
    \draw[-] (K6a) -- (K6b) -- (K6d);
    \draw[-] (K6b) -- (K6i) -- (K6g);
  \node at (0,-2) {$\alg{IS}_6$};
\end{tikzpicture}
\hspace{8pt}
\begin{tikzpicture}[scale=1,
  dot/.style={circle,fill,inner sep=1.6pt,outer sep=2pt}]
  \node (K6a) at (1,-1) {$\bbot$};
  \node (K6b) at (1,0) {$0 $};
    \node (K6g) at (1,2) {$1$};
        \node (K6h) at (1,3) {$\ttop$};
            \node (K6i) at (1,1) {$a$};
  \draw[-] (K6g) -- (K6h);
    \draw[-] (K6a) -- (K6b) 
    ;
    \draw[-] (K6b) -- (K6i) -- (K6g);
  \node at (1,-2) {$\alg{IS}_5$};
\end{tikzpicture}
\hspace{8pt}
\begin{tikzpicture}[scale=1,
  dot/.style={circle,fill,inner sep=1.6pt,outer sep=2pt}]
  \node (K6a) at (1,-1) {$\bbot$};
  \node (K6b) at (1,0) {$0$};
    \node (K6g) at (1,1) {$1$};
        \node (K6h) at (1,2) {$\ttop$};
  \draw[-] (K6g) -- (K6h);
    \draw[-] (K6a) -- (K6b) 
    ;
    \draw[-] (K6b) -- (K6i) -- (K6g);
  \node at (1,-2) {$\alg{IS}_4$};
\end{tikzpicture}
\hspace{8pt}
\begin{tikzpicture}[scale=1,
  dot/.style={circle,fill,inner sep=1.6pt,outer sep=2pt}]
  \node (K6a) at (1,-1) {$\bbot$};
  \node (K6b) at (1,0) {$0 $};
        \node (K6h) at (1,1) {$\ttop$};
  \draw[-] (K6g) -- (K6h);
    \draw[-] (K6a) -- (K6b) 
    ;
    \draw[-] (K6b) -- (K6i) -- (K6g);
  \node at (1,-2) {$\alg{IS}_3$};
\end{tikzpicture}
\end{center}
\end{figure}

The variety of IS-algebras is generated by the six-element algebra $\alg{IS}_6$, which is shown in Figure~\ref{fig:1}
together with its subalgebras $\alg{IS}_5$, $\alg{IS}_4$ and $\alg{IS}_3$. 
Our notation reflects the observation that 
the De Morgan algebra reduct of $\alg{IS}_6$ is obtained by adjoining a new top $\ttop$ and a new
bottom $\bbot$ element to the De Morgan algebra $\alg{DM}_4$, and by extending the De Morgan operations
in the obvious way (in particular, $\nnot \ttop = \bbot$ and $\nnot \bbot = \ttop$); cf.~Proposition~\ref{prop:newtop}.
Likewise, $\alg{IS}_5$ is obtained from $\alg{K_3}$, 
and $\alg{IS}_4$ from $\alg{B_2}$ (we can view $\alg{IS}_3$ as obtained in the same way if we start from the 
one-element trivial  De Morgan algebra); this observation will be central in our
approach to IS-algebras (see Section~\ref{sec:sem}).
The operation $\nabla$ is given on $\alg{IS}_6$ (and on its subalgebras)
by $\nabla \bbot = \bbot $ and 
$\nabla 0 = \nabla a = \nabla b =  \nabla 1 = \nabla \ttop = \ttop$.
%
Observe that $\alg{IS}_2$ is 
isomorphic to the two-element Boolean algebra $\alg{B_2}$
(on which $\nabla$  is the identity map). It is also well known that,  upon defining $x \imp y := (\nabla \nnot x \lor y ) \land (\nabla y \lor \nnot x)$, the algebra $\alg{IS}_3$ can be endowed with an MV-algebra
structure~\cite[Thm.~2.9]{CiGa83}.
Conversely, on the three-element MV-algebra, 
one obtains an IS-algebra structure by letting
$\nabla x : = \nnot x \imp x$. Thus $\alg{IS}_3$ can be viewed as the three-element MV-algebra.
On the other hand, on $\alg{IS}_4$ and $\alg{IS}_5$ the \L ukasiewicz implication is not definable (to see this, it is
sufficient to observe that neither of these algebras is simple, whereas it is known that every finite MV-chain is~\cite[Cor.~3.5.4]{CiMuOt99}.
The following result is an easy consequence of the observations in~\cite{CiGa83}, and entails
that, to verify the validity on all IS-algebras of not only identities but also quasi-identities, it is sufficient to
test them on 
$\alg{IS}_6$.

 \begin{proposition}
\label{prop:quasgen}
$\ISA = \VV (\alg{IS}_6) = \QQ (\alg{IS}_6)$. 
\end{proposition}

\begin{proof}
That $\ISA = \VV (\alg{IS}_6)$ is well-known~\cite[Cor.~3.5]{CaFi18}.
%
A sufficient condition for having $\VV (\alg{IS}_6) = \QQ(\alg{IS}_6) = \II \SU \PR (\alg{IS}_6)$
is that all the subdirectly irreducible algebras in 
$\VV (\alg{IS}_6)$
be subalgebras of $\alg{IS}_6$ (see e.g.~\cite[Thm.~3.6.ii]{ClDa98}). The latter indeed holds, and has been observed in~\cite{CiGa83}; see
Thm.~2.8 therein and the subsequent remarks.
\end{proof}



The following easy observation will be very useful in our study of IS-logics from Section~\ref{sec:sem} on.
Given a De Morgan algebra $\Al 
$, let
$\Al^{\nabla} = \la A \cup \{ \bbot, \ttop \}, \nabla \ra$ be the algebra defined as follows.
The lattice reduct of $\Al^{\nabla}$ is obtained by adjoining a new top element $\ttop$ and a new bottom 
element $\bbot$ to the lattice reduct of $\Al$. The De Morgan negation $\nnot$ is extended to $\Al^{\nabla}$ in the obvious way,
i.e.~by letting $\nnot \ttop = \bbot$ and $\nnot \bbot = \ttop$. Furthermore, the unary operator $\nabla$ is defined as follows:
$\nabla \bbot = \bbot$ and $\nabla a = \nabla \ttop = \ttop$ for all $a \in A$.
It is then clear that the $\nabla$-free reduct of $\Al^{\nabla}$ is a De Morgan algebra, and it is very easy to check that
$\nabla$ satisfies all the properties required by Definition~\ref{def:dm}.

The following easy observation will be very useful in our study of IS-logics from Section~\ref{sec:sem} on.
Let $\Al = \la A; \land, \lor, \nnot, \bot, \top  \ra
$ be an algebra in the language of De Morgan algebras.
Given $\bbot, \ttop \notin A$, we define
the algebra
$\Al^{\nabla} = \la A \cup \{ \bbot, \ttop \}, \nabla \ra$ as follows:

$ {\nabla} x := 
\begin{cases}
 \bbot &\mbox{ if }x=\bbot \\
 \ttop&\mbox{ otherwise }
\end{cases}
$  \qquad
$ {\nnot} x := 
\begin{cases}
 \nnot_{\A} x &\mbox{ if }x\in A\\
 \bbot &\mbox{ if }x=\ttop\\
 \ttop&\mbox{ if }x=\bbot 
\end{cases}
$ 
%

$x  {\land} y := 
\begin{cases}
 x\land_{\A} y &\mbox{if }x,y\in A\\
 \ttop&\mbox{if }x=y=\ttop\\
 \bbot &\mbox{if }x=\bbot \mbox{ or }y=\bbot \\
 z &\mbox{if }\{x,y\}=\{\ttop,z\},z\in A
\end{cases} 
$
 \quad
$x  {\lor} y := 
\begin{cases}
 x\lor_{\A} y &\mbox{if }x,y\in A\\
 \bbot &\mbox{if }x=y=\bbot \\
 \ttop&\mbox{if }x=\ttop\mbox{ or }y=\ttop\\
  z &\mbox{if }\{x,y\}=\{\bbot,z\},z\in A
\end{cases} 
$

$$\top := \ttop \quad \quad \bot := \bbot.$$ 

In case $\A$ is a De Morgan algebra, it is clear that the $\nabla$-free reduct of $\Al^{\nabla}$ is also a De Morgan algebra. 
Moreover, it is very easy to check that the above-defined
$\nabla$ operation satisfies all the properties required by Definition~\ref{def:dm}.
Thus, we have the following.

 \begin{proposition}
\label{prop:newtop}
For every De Morgan algebra $\Al 
$,
the above-defined algebra $\Al^{\nabla}$ is an IS-algebra.
\end{proposition}



\section{Semantical considerations on  IS-logics}
\label{sec:sem}

It is shown in~\cite[Thm~5.2]{CaFi18}
the order-preserving logic of the variety
$\ISA$
(which we denote by $\ISL$)
coincides with the logic determined
by the closure system of all lattice filters on the generating algebra $\alg{IS}_6$ (which are the principal up-sets
$\uparrow \! 0, \uparrow \! a, \uparrow \!  b, \uparrow \! 1$, and $\{ \ttop \}$).
Since
the matrices
$\la \alg{IS}_6, \uparrow \! a \ra $ and 
$\la \alg{IS}_6, \uparrow \! b \ra $ define the same logic~\cite[Lemma~5.4]{CaFi18}, we have that 
$\ISL$ is determined by the following set of matrices:
$\{ \la \alg{IS}_6, \uparrow \! 0 \ra , \la \alg{IS}_6, \uparrow \!  a \ra,  \la \alg{IS}_6, \uparrow \! 1 \ra, \la \alg{IS}_6,  \{  \ttop \} \ra \}$.
This result can be further sharpened, as the following Proposition shows.

 \begin{proposition}
\label{prop:onemat}
$\ISL = 
\Log{\la \alg{IS}_6, \uparrow \!  a \ra}$.
\end{proposition}

\begin{proof}
Observe that the matrices
$ \la \alg{IS}_6, \uparrow \! 1 \ra $ and
$\la \alg{IS}_6, \uparrow \!  a \ra$ are reduced, 
while
$\la \alg{IS}_6, \{ \ttop\} \ra $
and
$\la \alg{IS}_6, \uparrow \! 0 \ra$ are not.
The reduction of $\la \alg{IS}_6,  \{ \ttop\} \ra $ 
 is isomorphic to $\la \alg{IS}_3,  \{ \ttop\} \ra $, and is therefore isomorphic
to a submatrix of $\la \alg{IS}_6, \uparrow \!  b \ra$. 
The reduction of $\la \alg{IS}_6, \uparrow \! 0 \ra$ is isomorphic to 
$\la \alg{IS}_3,  \{ 0, \ttop\} \ra $, which in turn is isomorphic to  a submatrix
of $ \la \alg{IS}_6, \uparrow \!  a \ra$.
Thus 
$
\Log{\{ \la \alg{IS}_6, \uparrow \!  a \ra, \la \alg{IS}_6, \uparrow \! 1 \ra, \la \alg{IS}_6,  \{ \ttop\} \ra, \la \alg{IS}_6, \uparrow \! 0  \ra \}}=
\Log{\{ \la \alg{IS}_6, \uparrow \!  a \ra, \la \alg{IS}_6, \uparrow \! 1 \ra \}}
$.
To conclude the proof,
it suffices to show that
$\Log{\la \alg{IS}_6, \uparrow \!  a \ra \subseteq \Log \la \alg{IS}_6, \uparrow \! 1 \ra}  $.
To see this,
notice that $\uparrow \! 1 = \uparrow \! a \ \cap \uparrow \! b$.
This easily entails that 
$\Log {\{ \la \alg{IS}_6, \uparrow \!  a \ra, \la \alg{IS}_6, \uparrow \!  b \ra \}} \subseteq \Log{\la \alg{IS}_6, \uparrow \! 1 \ra}  $, and we have seen that
$\Log {\{ \la \alg{IS}_6, \uparrow \!  a \ra, \la \alg{IS}_6, \uparrow \!  b \ra \}} = \Log{\la \alg{IS}_6, \uparrow \!  a \ra} $.
\end{proof}

In our study, it will be useful to be able to work with \emph{reduced} matrix models of IS-logics. 
Proposition~\ref{prop:nonprot} below suggests that these cannot be characterized by simply applying the Blok-Pigozzi
algebraization process, but Proposition~\ref{prop:alg}
provides  sufficient information  for our purposes. 
(For the definitions of selfextensional, protoalgebraic and algebraizable logic,
we refer the reader to~\cite{F16}, respectively, Def.~ 5.25,~6.1 and~3.11).

 \begin{proposition}
\label{prop:nonprot}
$\ISL$
is selfextensional and non-protoalgebraic (hence, non-algebraizable).
\end{proposition}

\begin{proof}
Selfextensionality simply follows from the observation that two formulas $\varphi, \psi$ are inter-derivable
in $\ISL$
if and only if the identity $\phi \approx \psi$ is valid in the variety of IS-algebras.
To show that our logic is not protoalgebraic, we verify that the Leibniz operator
$\Leibniz$
is not monotone on matrix models~\cite[Thm.~6.13]{F16}. 
To see this, observe that the algebra $ \alg{IS}_6$ has (exactly) one non-trivial congruence $\theta$,
which identifies  the elements $\{ 0, a, b, 1\}$. It is then easy to check that 
$\Leibniz_{ \alg{IS}_6} ( \{ \ttop \}) = \theta$.
On the other hand, the matrix $\la \alg{IS}_6,   \uparrow \! 1 \ra$ is reduced. Hence, the Leibniz operator is not monotone
on the matrix models based on $\alg{IS}_6$.  
\end{proof}

\begin{remark}
Proposition~\ref{prop:nonprot} can in fact be slightly strengthened. If we consider the matrices
 $\la \alg{IS}_4,  \{ \ttop \} \ra$ and  $\la \alg{IS}_4,  \uparrow \! 1  \ra$, where 
$\alg{IS}_4$ is the four-element subalgebra of $\alg{IS}_6$ with universe $\{ \bbot, 0, 1, \ttop \}$, we can observe that
$\la \alg{IS}_4,   \uparrow \! 1 \ra$ is reduced while $\la \alg{IS}_4,  \{  \ttop \} \ra$ is not. Hence the logic determined by these two
submatrices (which is obviously stronger than $\ISL$) 
 is also non-protoalgebraic. This, in turn,
entails that $\ISL$
cannot be protoalgebraic. 
\end{remark}

The following result is an instance of a general result
on order-preserving logics (see e.g.~\cite[Thm.~2.13.iii]{AlPrRi16}).

 \begin{proposition}
\label{prop:alg}
$\Alg (\ISL) = \ISA$. 
\end{proposition}


The next Proposition we characterizes 
the logic determined by the class of matrices 
$\{ \la \Al, \{ 1 \} \ra : \Al \in \ISA \}$.
The latter 
(denoted by $\vdash^{1}_{\ISA} $)
is known in the algebraic logic literature as the \emph{1-assertional logic}
of the class $\ISA$.

 \begin{proposition}
\label{prop:quas}
${\vdash^{1}_{\VV (\Al[IS]_3) }} = {\vdash^{1}_{\ISA}} = \Log \la \Al[IS]_3, \{1 \} \ra$. 
\end{proposition}

\begin{proof}
Obviously  ${\vdash^{1}_{\VV (\Al[IS]_3) }} \subseteq {\vdash^{1}_{\ISA}} 
\subseteq \Log \la \Al[IS]_3, \{1 \} \ra$, so it suffices to verify the inequality
$  \Log \la \Al[IS]_3, \{1 \} \ra \subseteq
{\vdash^{1}_{\VV (\Al[IS]_3) }}
$. 
Assume $\Gamma \vdash_{\Log \la \Al[IS]_3, \{1 \} \ra} \varphi $. Observe that, since $\Log \la \Al[IS]_3, \{1 \} \ra$ is finitary, we can assume
$\Gamma$ to be finite. Then 
$\gamma   \vdash_{\Log \la \Al[IS]_3, \{1 \} \ra} \varphi$ for  $\gamma : = \bigwedge \Gamma$. The latter is equivalent to 
$\gamma   \vdash_{\Log \la \Al[IS]_6, \{1 \} \ra} \varphi$, because, as observed earlier,
$\la \Al_6, \{1 \} \ra^* =  \la \Al[IS]_3, \{1 \} \ra  
$. In turn, $\gamma   \vdash_{\Log \la \Al_6, \{1 \} \ra} \varphi$
entails that $\Al[IS]_6$ satisfies the quasi-identity $\gamma \approx \top \imp \varphi \approx \top$. By Proposition~\ref{prop:quasgen},
this entails that $\gamma \approx \top \imp \varphi \approx \top$ is satisfied by every
$\A \in \ISA$.
Hence,  $\gamma \vdash \varphi$ holds in every matrix in the class $\{ \la \Al, \{ 1 \} \ra : \Al \in \ISA 
\}$ and, a fortiori, 
in every matrix in the class $\{ \la \Al, \{ 1 \} \ra : \Al \in \VV(\Al[IS]_3) 
\}$.
This means that $\gamma  \vdash^{1}_{\Al[IS]_3} \varphi$ or, equivalently,
$\Gamma \vdash^{1}_{\Al[IS]_3} \varphi$.
\end{proof}

Recalling that the algebra $\Al[IS]_3$ is isomorphic to 
to the
three-element \L ukasiewicz(-Moisil) algebra,
Proposition~\ref{prop:quas} tells us 
 that  
$\vdash^{1}_{\ISA} $ is (term equivalent) to  three-valued 
\L ukasiewicz logic. This logic is axiomatized, relatively to $\ISL$, in
Theorem~\ref{th:axnotinshape} (i).


We now return to the construction introduced at the end of Section~\ref{sec:isal}
and  illustrate its remarkable logical consequences.
Let $\Mt = \la \A, D \ra$ be a matrix, with $\A$ an algebra in the language of De Morgan algebras.
Then $\A^{\nabla}$ is in the language of $\ISA$. 
Denoting by $\widehat{\A}$ the $\nabla$-free reduct of $\A^{\nabla}$,
 let $\widehat{\Mt} := \la \widehat{\Al}, D \cup \{ \ttop \} \ra$.


 \begin{lemma}
\label{lem:newtopmat}
Let  $\Mt 
$ be a reduced model of $\B$. 
Then $\Mt \cong (\widehat{\Mt})^{*} $.
\end{lemma}

\begin{proof}
Recall from~\cite[Thm.~3.14]{F97} that all reduced models of $\B$
are matrices $\Mt = \la \A, F \ra$, with $\A$ a De Morgan algebra and $F$ a lattice filter.
To establish the claim, it suffices to show that $\Leibniz_{\widehat {\Al}} (F \cup \{ \ttop \}) = Id_{\widehat {A}} \cup \{ \la 0, \bbot \ra, \la \bbot, 0 \ra, \la 1, \ttop \ra, \la \ttop, 1 \ra \}$. Let $a,b \in \widehat {A}$. According to~\cite[Prop.~3.13]{F97}, we have $\la a,b \ra \in \Leibniz_{\widehat {\Al}} (F \cup \{ \ttop \})$
if and only if, for all $c \in \widehat {A}$, the following hold: ($a \lor c \in F \cup \{ \ttop \}$ iff $b \lor c \in F \cup \{ \ttop \}$)
and ($\nnot a \lor c \in F \cup \{ \ttop \}$ iff $\nnot b \lor c \in F \cup \{ \ttop \}$). Let us show that $\la 1, \ttop \ra \in \Leibniz_{\widehat {\Al}} (F \cup \{ \ttop \})$. Observe that $1 \lor c, \ttop \lor c \in F \cup \{ \ttop \}$ for all $c \in \widehat {A}$. The first condition is thus obviously satisfied. As to the second, assume
$\nnot 1 \lor c = 0 \lor c \in F \cup \{ \ttop \}$ for some $c \in \widehat {A}$. Then $c \notin \{ 0, \bbot \}$, because $0 \lor \bbot = 0 \lor 0 = 0 \notin F \cup \{ \ttop \}$. This entails $0 < c$, so $0 \lor c = c \in F \cup \{ \ttop \}$. Hence, $\nnot \ttop \lor c \in F \cup \{ \ttop \}$. 
Conversely, if $\nnot \ttop \lor c = \bbot \lor c \in F \cup \{ \ttop \}$, then 
we immediately have $\nnot 1 \lor c = 0 \lor c \in F \cup \{ \ttop \}$ because $\bbot \leq 0$.
Hence, $\la 1, \ttop \ra \in \Leibniz_{\widehat {\Al}} (F \cup \{ \ttop \})$. By the congruence properties, this entails 
$\la \nnot 1, \nnot \ttop \ra= \la 0, \bbot \ra \in  \Leibniz_{\widehat {\Al}} (F \cup \{ \ttop \})$, thus also
$\la \bbot, 0 \ra, \la \ttop, 1 \ra \in  \Leibniz_{\widehat {\Al}} (F \cup \{ \ttop \})$.
Now let $\la a,b \ra \in  \Leibniz_{\widehat {\Al}} (F \cup \{ \ttop \})$ be such that
$a, b \notin \{ \bbot, \ttop \}$ and $a \neq b$. The latter assumption entails that 
$\la a,b \ra \notin  \Leibniz_{\Al} (F )$, because $\Mt$ was reduced. Then, by~\cite[Prop.~3.13]{F97}, there is $c \in A$
such that either ($a \lor c \in F$ and $b \lor c \notin F$) or ($\nnot a \lor c \notin F$ and $ \nnot b \lor c \in F$).
In the former case, we have $a \lor c \in F \cup \{ \ttop \}$ and $b \lor c \notin F \cup \{ \ttop \}$, because $b \lor c \neq \ttop$ for 
all $b, c \in A$. Hence, we should have $\la a,b \ra \in  \Leibniz_{\widehat {\Al}} (F \cup \{ \ttop \})$, contradicting our assumptions.
A similar reasoning shows that the latter case also leads to a contradiction. 
\end{proof}

We  note that  Lemma~\ref{lem:newtopmat} 
could be proved in a more general form, that we shall however not need for our present purposes. Indeed, given an arbitrary (not necessarily reduced) model $\Mt 
$ of $\B$, 
one can show that 
$\Mt^{*} \cong (\widehat {\Mt})^{*} $.

 \begin{corollary}
\label{cor:newtopmat}
Given matrices $\Mt $ and $ \widehat {\Mt} $ as per Lemma~\ref{lem:newtopmat}, 
we have $\Log \Mt  = \Log{\widehat {\Mt}}$.
\end{corollary}

As before,  $\Mt = \la \A, F \ra$ is a matrix such that  $\Al$ is a De Morgan algebra and $F \subseteq A$ a lattice filter on $\Al$.
Consider  the IS-algebra $
{\Al}^{\nabla}$  defined according to Proposition~\ref{prop:newtop},
and let $\Mt^{\nabla} = \la 
{\Al}^{\nabla}, F \cup \{ \ttop \} \ra$.

 \begin{corollary}
\label{cor:newtopconsex}
Let  $\Mt = \la \A, F \ra$ be a reduced matrix, with $\Al$ a De Morgan algebra and $F \subseteq A$ a lattice filter on $\Al$.
Then $\Log\Mt^{\nabla}$ 
is a conservative expansion of $\Log \Mt$.
\end{corollary}

\begin{proof}
Using Corollary~\ref{cor:newtopmat}, it suffices to observe that the $\nabla$-free fragment of $\Log \Mt^{\nabla} $ is $\Log {\widehat{\Mt}} $.
\end{proof}


 \begin{corollary}
\label{cor:forlogbel}
$\ISL$ is a conservative expansion of the Belnap-Dunn
logic $\B$.
\end{corollary}

\begin{proof}
Recall that
$\ISL = \Log {\la \alg{IS}_6, \uparrow \!  a \ra} $ (Proposition~\ref{prop:onemat}), and
observe that the matrix
$\la \alg{IS}_6, \uparrow \!  a \ra$ can be obtained as $\Mt^{\nabla}$
from the four-element matrix $\Mt = \la \alg{DM}_4, \uparrow \! a \ra $ that defines $\B$. Then the result follows from
Corollary~\ref{cor:newtopconsex}.
\end{proof}

Recall that all reduced matrices for the Belnap-Dunn logic (hence, also all reduced matrices for super-Belnap logics)
are of the form $\la \A, F \ra$, with $\A$ a De Morgan algebra and $F$ a lattice filter~\cite[Thm.~3.14]{F97}.
Thus, Lemma~\ref{lem:newtopmat} and
Corollaries~\ref{cor:newtopmat} and~\ref{cor:newtopconsex} apply, and the latter gives us  the following result.

 \begin{corollary}
\label{cor:twomod}
Let  $\Mt_1 = \la \A_1, F_1 \ra$ and $\Mt_2 = \la \A_2, F_2 \ra$ be reduced matrices for the Belnap-Dunn logic.
If $\Log {\Mt^{\nabla}_1} = \Log {\Mt^{\nabla}_2} $, then $\Log \Mt_1  = \Log \Mt_2 $.
\end{corollary}

Given a super-Belnap logic  $\Lo$, let 
$\Lo^{\nabla}
= \Log {\{  \Mt^{\nabla} :  \Mt \in \Matr^* (\Lo) \}}
$,
%
where each $\Mt^{\nabla} = \la 
{\Al}^{\nabla}, F \cup \{ \ttop \} \ra$
is defined as before.
Since $F \cup \{ \ttop \}$ is a lattice filter of $ 
{\Al}^{\nabla}$, 
every $\Mt^{\nabla}$ is a model of 
$\ISL$.
Therefore,
each logic $\Lo^{\nabla}$ is an extension of $\ISL$.

 \begin{lemma}
\label{lem:consex}
Let  $\Lo$ be a super-Belnap logic. Then $\Lo^{\nabla}$ is a conservative expansion of 
$\Lo$.
\end{lemma}

\begin{proof}
Suppose, in view of a contradiction, that there exist formulas $\Gamma, \phi$ 
in the $\nabla$-free language
such that
$\Gamma  \vdash \phi$ holds in  $\Lo^{\nabla}$ but
does not hold in $\Lo$.
Then there is $\Mt \in  \Matr^* (\Lo)$ such that 
$\Gamma  \not\vdash_{\Mt} \phi$. 
Then, 
by Corollary~\ref{cor:newtopconsex}, we have that 
$\Gamma  \not\vdash_{\Mt^{\nabla}} \phi$. 
By definition, $\Lo^{\nabla} \subseteq \Log \Mt^{\nabla} $.
Hence, $\Gamma  \not\ders^{\nabla} \phi$. 
\end{proof}

 \begin{corollary}
\label{cor:forlogs}
Let $\Lo_1$, $\Lo_2$ be super-Belnap logics. Then
${\Lo_1} \subseteq {\Lo_2}$ if and only if
${\Lo^{\nabla}_1} \subseteq {\Lo^{\nabla}_2}$.
\end{corollary}

\begin{proof}
Assuming ${\Lo_1}  \subseteq {\Lo_2} $, we have 
$
\Matr^* (\Lo_2) \subseteq  \Matr^* (\Lo_1)
$. Hence,  
$
\{  \Mt^{\nabla} :  \Mt \in \Matr^* (\Lo_2) \} \subseteq \{  \Mt^{\nabla} :  \Mt \in \Matr^* (\Lo_1) \}
$, which entails 
${\Lo^{\nabla}_1} \subseteq {\Lo^{\nabla}_2}$.
Conversely, 
let ${\Lo^{\nabla}_1} \subseteq {\Lo^{\nabla}_2}$, and let $\Gamma, \phi$ be formulas
(in the language of $\B$
) such that
$\Gamma \ders_1 \phi$. 
The latter assumption 
gives us that 
$\Gamma \Lo^{\nabla}_1 \phi$ 
and, therefore,  
also  $\Gamma \Lo^{\nabla}_2 \phi$. Then, by Lemma~\ref{lem:consex}, we conclude $\Gamma \Lo_2 \phi$. 
\end{proof}

 \begin{corollary}
\label{cor:latlog}
The map given by $\Lo \ \mapsto \ \Lo^{\nabla}$ is 
 an embedding of the lattice of super-Belnap logics
 into the lattice of extensions of $\ISL$.
\end{corollary}

 \begin{corollary}
\label{cor:cardlat}
The lattice of extensions of $\ISL$ has (at least) the cardinality of the continuum.
\end{corollary}

\begin{proof}
By Corollary~\ref{cor:latlog} and the observation that the lattice of super-Belnap logics
contains continuum many logics
~\cite[Thm.~4.13]{AlPrRi16}.
\end{proof}

In fact, in the light of the results of Section~\ref{sec:ax}, we shall be able to 
prove that
there are at least continuum many \emph{finitary} extensions of  $\ISL$.

\section{Axiomatizing IS-logics}
\label{sec:ax}

{

In~\cite{Ca-MSc,CaFi18}, the logic $\ISL$ is axiomatized
by means of a Gentzen calculus.  
In this Section we tackle the problem of 
axiomatizing $\ISL$ and its extensions by means of Hilbert calculi.
From a technical point of view, we shall take profit from the theory of
multiple-conclusion calculi,   a
generalization of 
traditional Hilbert-style calculi
in which 
 the 
inference rules can have more than one conclusion (with a disjunctive reading). 
In  these calculi proofs are typically ramified instead of sequential. 
Multiple-conclusion
calculi can be used to study single conclusion logics, but also correspond to a generalized notion of logic due to 
D.~Scott and developed by D.J.~Shoesmith and T.J.~Smiley.
We recall some of the basic definitions and results below; for further details see~\cite{ShSm78,CaMaXX}.

A \emph{multiple-conclusion consequence relation} (logic) is a relation ${\derm} \subseteq \wp Fm\times \wp Fm$ satisfying the following conditions.
For every $\Gamma,\Gamma',\Delta,\Delta',\Lambda, T, F \subseteq Fm$,
\begin{enumerate}[(i)]
\item $\Gamma\,\derm\, \Delta$ whenever $\Gamma\cap \Delta\neq \emptyset$ (\emph{overlap}),
\item $\Gamma,\Gamma'\,\derm\, \Delta,\Delta'$ whenever $\Gamma\,\derm\,\Delta$ 
 (\emph{dilution}),
\item  $\Gamma\,\derm\, \Delta$ whenever $\Gamma, T\,\derm\, \Delta, F$ 
 for every partition $\tuple{T,F}$ of $\Lambda$ 
 (\emph{cut for sets}),
\item  $\Gamma^\sigma\,\derm\, \Delta^\sigma$ for every substitution $\sigma$ whenever $\Gamma\,\derm\, \Delta$  (\emph{substitution invariance}).
\end{enumerate}

Given 
a set of multiple-conclusion rules $\R\subseteq \wp Fm\times \wp Fm$, we denote by 
$\derm_\R$  the smallest multiple-conclusion consequence relation containing $\R$
(hence, 
$\R$ axiomatizes $\derm_\R$). 
From a proof-theoretic perspective, we have  $\Gamma\derm_\R\Delta$ whenever there is a 
{labelled tree-proof } 
whose root is labelled by $\Gamma$ and the leaf of every non-discontinued branch is labelled with a formula in $\Delta$. 
Every class of matrices $\CMt$ determines a multiple-conclusion logic
defined as follows:
we let $\Gamma\,\derm_\CMt\, \Delta$ whenever, for every valuation $v \in \Val(\Mt)$ over a matrix  $ \Mt = \la \A, D \ra \in \CMt$, we have  that $v(\Gamma)\subseteq D$ implies $v(\Delta)\cap D\neq \emptyset$.


Multiple-conclusion logics smoothly generalize Tarskian logics and their proof-theoretic and semantical definitions.
Indeed, 
for every multiple-conclusion logic $\derm$,
we have that 
${\ders_\derm} = {\derm} \cap (\wp L\times L)$ is a Tarskian consequence relation~\cite[Def.~1.5]{F16}. 
We call $\ders_\derm$ the \emph{single-conclusion companion} of $\derm$ and, given a set
of multiple-conclusion rules $\R$, we shall write $\ders_\R$ instead of $\ders_{\derm_\R}$.
%
The following remark contains a few useful facts that can be easily deduced from Sections~5.2 and~17.3 
  of~\cite{ShSm78}.
\begin{remark}\label{rem:svsm}
The $\emph{sign}$ of a multiple-conclusion relation $\derm$ is  \emph{negative} if 
 {\color{blue}$Fm\derm \emptyset$}. 
and is \emph{positive} otherwise. 
We denote by $\simeq $ the equivalence relation that identifies two logics $\derm_1$  and $\derm_2$ that may differ only in the sign, that is, we let
 ${\derm_1}\simeq {\derm_2} $ whenever 
 {\color{blue}${ \derm_1 } \cup\{(Fm,\emptyset)\}= {\derm_2} \cup \{(Fm,\emptyset)\}$.}
Let ${\derm_1} \simeq {\derm_2}$. Then
 $\ders_{\derm_1}= \ \ders_{\derm_2}$ and also, if 
  ${\derm_1}\subseteq {\derm} \subseteq {\derm_2}$, then ${\derm_1} \simeq {\derm_2} \simeq {\derm}$.
 Let $\PR (\CMt)$ be the closure under 
 products of the class $\CMt$ (products among matrices are defined as  usual for first-order structures; see e.g.~\cite[p.~225-6]{F16}).
The following observations are well known:
 \begin{enumerate}[(i)]
 \item 
 $\ders_{\derm_\CMt} = \Log \CMt =\Log \PR(\CMt) $. 
 \item 
If
 $\Log \CMt = {\ders_\R}$  for $\R\subseteq \wp(L)\times L$, then   ${\derm_\R} \simeq {\derm_{\PR(\CMt)}}$. Therefore, if 
 $\Log \CMt_1 =\Log \CMt_2 $ then ${\derm_{\PR(\CMt_1)} } \simeq { \derm_{\PR(\CMt_2)}} $.
 \end{enumerate}
 \end{remark}
 
%
%
%
%
%
%
%
%
%


Under certain conditions, a (finite) single-conclusion axiomatization can be obtained algorithmically  from a (finite) multiple-conclusion axiomatization.
The following result covers 
the case 
of some of the logics that interest us here. 
%
%
%


Given a finite set $\Phi = \{ \phi_1, \ldots, \phi_n \} \subseteq Fm$ and $\psi\in Fm$,  
 let $\bigvee\Phi := ((\phi_1 \lor \phi_2  ) \lor \ldots )\lor  \phi_n $,
 and let
$\Phi\lor \psi =\{\varphi\lor \psi:\varphi\in \Phi\}$.


\begin{theorem}\cite[Thm. 5.37]{ShSm78}
\label{disjax}
Let   $\R$ be a set of multiple-conclusion rules.
Suppose $\ders_{\derm_\R}$ satisfies, for all $\Gamma \cup \{\varphi, \psi, \xi \} \subseteq Fm$, the following property: $\Gamma,\varphi\lor \psi\, {\ders_{\derm_\R}}\xi$ if and only if 
$\Gamma,\varphi\, {\ders_{\derm_\R}}\xi$ and $\Gamma,\psi\, {\ders_{\derm_\R}}\xi$. Then $\ders_{\derm_\R}$ is axiomatized by the set $\R^\lor$
consisting of the following rules:
\begin{enumerate}[(i)]
\item $\mathsf{r}^\lor=\frac{}{\varphi}$ for each  $\mathsf{r}=\frac{}{\varphi}\in \R$, 
 \item $\mathsf{r}^\lor=\frac{\Gamma\lor p_0}{(\bigvee \Delta)\lor p_0}$ for each $\mathsf{r}=\frac{\Gamma}{\Delta}\in \R$,
  \item $\frac{p}{p\lor q}$,  $\frac{p\lor q}{q\lor p}$,  $\frac{p\lor p}{p}$ and $\frac{p\lor (q\lor r)}{(p\lor q)\lor r}$
\end{enumerate}
where $p_0$ is a variable not occurring in $\R$.
 
\end{theorem}


We now proceed to explain how the results of Section~\ref{sec:sem} together with the general considerations on multiple-conclusion
logics introduced above are going to be help us deal with extensions of
$\ISL$.

\subsection{Adding $\nabla$ to the Belnap-Dunn logic
}

Let 
$\Sigma = \{ \land, \lor, \nnot , \bot, \top \}$
be 
the language of $\B$, and
let 
$\Sigma^{\nabla}$ be the expansion of $\Sigma$ with the unary connective $\nabla$
(i.e.~the language of $\ISL$ 
). 
Given a $\Sigma$-matrix 
$\Mt=\tuple{\A, D}$, 
let ${\Mt}^\nabla=\tuple{ \A^{\nabla}, 
D\cup\{ \ttop \}}$ 
be the $\Sigma^{\nabla}$-matrix
with underlying algebra 
$\A^{\nabla}$ defined as in Section~\ref{sec:isal} (cf.~Propositiion~\ref{prop:newtop}).
Let us denote by $\widehat{\Mt} $
the $\Sigma$-fragment 
of $\Mt^{\nabla}$.
Observe that, if  $\Mt = \la \A, D \ra $ with $\A$ a De Morgan algebra,
then $\widehat{\Mt} $ is precisely  the matrix 
considered in Corollary~\ref{cor:newtopmat}.
Given a class of $\Sigma$-matrices $\CMt$, we let
%
$\CMt^\nabla :=\{\Mt^\nabla:\Mt\in \CMt\}$ and $\widehat{\CMt} : =\{ \widehat{\Mt} :\Mt\in \CMt\}$.

The following Theorem contains a generic recipe for axiomatizing the multiple-conclusion
logic determined by the class $\CMt^\nabla$, assuming we have a set of rule $\R$
that axiomatizes the multiple-conclusion logic determined by $\widehat{\CMt}$.

\begin{theorem}\label{nablaax}

Let  $\CMt$ be a class of $\Sigma$-matrices.
If ${\derm_{\widehat{\CMt}}} \simeq {\derm_{\R}}$, 
 then ${\derm_{\CMt^\nabla}} = {\derm_{\R\cup \R_\nabla}} $, 
 where $\R_\nabla$ consists of the following rules:

 $$  \frac{}{\nabla p\,,\, \nnot \nabla p}  {\ \mathsf{r}_1} \qquad
   { \frac{\nabla p   }{\nnot\nabla\nnot\nabla  p}}  {\ \mathsf{r}_2} \qquad
    { \frac{ \nabla\nabla  p}{ \nabla p  }}  {\ \mathsf{r}_3} \qquad
   \frac{\nabla p\,,\, \nnot\nabla p}{}  {\ \mathsf{r}_4} 
  $$  
  
  $$
   { 
   \frac{\nnot\nabla p}{\nabla \nnot p}}  {\ \mathsf{r}_5}\qquad
   \frac{\nabla \nnot\nnot p }{\nabla p} {\ \mathsf{r}_6} \qquad 
      \frac{\nabla p}{\nabla \nnot\nnot p } {\ \mathsf{r}_7}$$

  $$
  \frac{\nabla(p\land q)}{\nabla p} {\ \mathsf{r}_8}\qquad 
 \frac{\nabla(p\land q) }{\nabla q} {\ \mathsf{r}_9} \qquad
\frac{\nabla \nnot (p\land q)}{\nabla \nnot p,\nabla \nnot q} {\ \mathsf{r}_{10}}
$$

 $$ 
  {
  \frac{ \nabla \nnot p}{ \nabla \nnot (p\land q) }} {\ \mathsf{r}_{11}} 
  \qquad 
  {
  \frac{\nabla \nnot q}{\nabla \nnot (p\land q) }}{\ \mathsf{r}_{12}} 
  \qquad 
  {
  \frac{ \nabla p \,,\,  \nabla q}{ \nabla  (p\land q)} }{\ \mathsf{r}_{13}}$$ 

 $$
 \frac{\nabla \nnot (p\lor  q)}{\nabla\nnot p}{\ \mathsf{r}_{14}}\qquad 
 \frac{\nabla \nnot (p\lor q)}{\nabla\nnot q}{\ \mathsf{r}_{15}} 
 \qquad \frac{\nabla(p\lor  q)}{\nabla  p,\nabla  q}{\ \mathsf{r}_{16}} 
 $$ 


$$
  \frac{ \nabla   p  }{  \nabla (p\lor q)}{\ \mathsf{r}_{17}}
  \qquad   \frac{   \nabla   q}{  \nabla (p\lor q)}{\ \mathsf{r}_{18}} 
  \qquad
  \frac{ \nabla \nnot  p \,,\,   \nabla \nnot  q}{  \nabla \nnot (p\lor q)} 
  {\ \mathsf{r}_{19}} 
$$

$$\frac{}{\nnot \nabla \bot}
\ \mathsf{r}_{20}
\qquad\frac{}{\nnot \nabla \nnot \top}
\ \mathsf{r}_{21}
 $$ 

\end{theorem}
\begin{proof}

Checking the soundness of the new rules is routine. We give only a couple of examples.
Let $v$ be a valuation over a matrix $\Mt^\nabla$.
The rule $\mathsf{r}_1$ is sound in $\Mt^\nabla$, for either $v(\nabla \varphi)=\ttop$ (if $v(\varphi)\neq \bbot$) or $v(\nnot \nabla \varphi)=\ttop$ (if $v(\varphi)= \bbot$).
Regarding rule $\mathsf{r}_2$, we have that, if $v(\nabla \varphi)=\ttop$ then $v(\nnot\nabla \varphi)=v(\nabla\nnot\nabla \varphi)=\bbot$, so $v(\nnot\nabla\nnot\nabla \varphi)=\ttop$.
  
 For completeness, assume 
 $\Gamma\not\derm_{\R_\nabla} \Delta$. Then, by cut for sets,
 there is a partition $\tuple{T,F}$ of $L_{\Sigma^\nabla}$ such that 
 $\Gamma \subseteq T$ and $\Delta\subseteq F$ 
 and $T\not\derm_{\R_\nabla}F$. 
 Note that (by $\mathsf{r}_1$ and $\mathsf{r}_4$) for each $\varphi$, we have either $\nabla \varphi\in T$ or $\nnot \nabla \varphi\in T$, but never both. In particular, $F$ is never empty.
Also, by $\mathsf{r}_5$ we must have either  $\nabla \varphi\in T$ or $\nabla\nnot  \varphi\in T$.
Hence, each $\varphi$ 
must be exactly in one of three cases:
(i) $\nabla \varphi,\nabla \nnot \varphi\in T$, (ii)  $\nnot\nabla \nnot \varphi\in T$, or (iii) $\nnot\nabla \varphi\in T$.

 Since $\R\subseteq \R\cup \R_\nabla$, we also have $T\not\derm_{\R}\ F$. 
  From the fact that $\derm_\R\simeq \derm_{\widehat{\Mt}}$ and $F\neq \emptyset$ we know that $T\not\derm_{\widehat{\Mt}}F$.   
 We can therefore pick $v\in \mathsf{Hom}_{\Sigma}(L_{\Sigma^\nabla},\widehat{\Mt})$, 
 for some $\Mt\in \CMt$ such that $v(T)\subseteq D$ and $v(F)\cap D=\emptyset$. 
 Consider $v':L_{\Sigma^\nabla}\to \Mt^\nabla$ defined by: 
$$v'(\varphi):=
\begin{cases}
\ttop&\mbox{ if }\nnot\nabla \nnot \varphi\in T\\
 \bbot&\mbox{ if }\nnot\nabla \varphi\in T\\
 v(\varphi_i)&\mbox{ if } \varphi=\varphi_1\land \varphi_2 \mbox{ and }\nnot\nabla \nnot \varphi_{3-i}\in T \\
 v\varphi_i)&\mbox{ if } \varphi=\varphi_1\lor \varphi_2 \mbox{ and } \nnot\nabla \varphi_{3-i}\in T\\
 v(\varphi )&\mbox{ if }\nabla \varphi,\nabla \nnot \varphi\in T\\
\end{cases}
$$

%
%
%



We will show that $v'\in \Val(\Mt^\nabla)=\mathsf{Hom}_{\Sigma^\nabla}(L_{\Sigma^\nabla},\widehat{\Mt})$.

\begin{enumerate}
\item $v'(\nabla \varphi)= \nabla v'(\varphi)$

From $\mathsf{r}_1$ we have that either (ii) $\nabla \varphi\in T$  or (iii) $\nnot \nabla \varphi\in T$. 

If (ii) $\nabla \varphi\in T$, by $\mathsf{r}_4$ we have that $\nnot\nabla \varphi\notin T$ and so $v'(\varphi)\neq \bbot$. 
Further, by $\mathsf{r}_2$ we obtain that $\nnot\nabla \nnot \nabla \varphi\in T$, hence $v'(\nabla \varphi)=\ttop= {\nabla}(v'(\varphi))$. 

If instead iii) $\nnot \nabla \varphi\in T$ then $v'(\varphi)=\bbot$ and by $\mathsf{r}_4$ we have that $\nabla \varphi\notin T$. Hence, by 
$\mathsf{r}_3$, $\nabla\nabla \varphi\notin T$, and by $\mathsf{r}_1$ we get that $\nnot \nabla\nabla \varphi\in T$ and 
$v'(\nabla \varphi)=\bbot= {\nabla}(v'(\varphi))$. 

\item $v'(\nnot \varphi)= \nnot v'(\varphi)$

If (i) $\nabla \nnot \varphi,\nabla \nnot \nnot \varphi\in T$ then by $\mathsf{r}_6$, $\nabla \varphi\in T$ (so $v'(\varphi)=v(\varphi)$) and therefore $v'(\nnot \varphi)=v(\nnot \varphi)=\nnot_\Mt(v(\varphi))=\tilde{\nnot}(v'(\varphi))$.

If (ii) $\nnot\nabla \nnot \nnot \varphi\in T$ (so $v'(\nnot\varphi)=\ttop$) then by $\mathsf{r}_1$ and $\mathsf{r}_7$ we have that $\nnot \nabla \varphi\in T$ (so $v'(\varphi)=\bbot$) hence  $v'(\nnot \varphi)=\ttop=\tilde{\nnot}(v'(\varphi))$.

If (iii) $\nnot\nabla \nnot \varphi\in T$ then $v'(\nnot \varphi)=\bbot$ and $v'(\varphi)=\ttop$, thus we immediately obtain $v'(\nnot \varphi)= \nnot v'(\varphi)$.

\item $v'(\varphi\land \psi)=v'(\varphi)  \land  v'(\varphi)$

If (i) we have that $\nabla (\varphi\land \psi),\nabla \nnot (\varphi\land \psi)\in T$. From $\nabla (\varphi\land \psi)\in T$, by 
$\mathsf{r}_8$ 
and $\mathsf{r}_{9}$
we obtain that $\nabla \varphi,\nabla \psi \in T$ 
($v'(\varphi)\neq\bbot\neq v'(\psi)$). Also, from $\nabla \nnot (\varphi\land \psi)\in T$, by 

$\mathsf{r}_{10}$,
either $\nabla\nnot \varphi$ or $\nabla \nnot \psi$ are in $ T$ ($v'(\varphi)\neq\ttop$ or $v'(\psi)\neq\ttop$). 
Hence, if $\nabla\nnot \varphi,\nabla \nnot \psi\in T$ then $v'(\varphi\land \psi)=v(\varphi\land \psi)=v(\varphi)\tilde{\land} v(\psi)=v'(\varphi)  \tilde{\land}  v'(\varphi)$.
Otherwise, if $\nabla\nnot \varphi\notin T$, then by $\mathsf{r}_{1}$ we conclude that  $\nnot\nabla\nnot \varphi\in T$ and hence
$v'(\varphi\land \psi)=v'(\psi)=\ttop\tilde{\land} v'(\psi)=v'(\varphi)\tilde{\land} v'(\psi)$. The case $\nabla\nnot \psi\notin T$ is similar.

If (ii)  we have that  $\nnot\nabla \nnot (\varphi\land \psi)\in T$  (so  $v'(\varphi\land \psi)=\ttop$) then by $\mathsf{r}_{4}$ we have that $\nabla \nnot (\varphi\land \psi)\notin T$.
By 
$\mathsf{r}_{11}$  and $\mathsf{r}_{12}$
we have that $\nabla \nnot \varphi, \nabla \nnot \psi\notin T$.
Hence, by $\mathsf{r}_{1}$, $\nnot \nabla \nnot \varphi,\nnot \nabla \nnot \psi\in T$ (so $v'(\varphi)=v'(\psi)=\ttop$) and  $v'(\varphi\land \psi)=\ttop=v'(\varphi)\tilde{\land} v'(\psi)$.

If (iii)  we have that  $\nnot\nabla (\varphi\land \psi)\in T$ (so $v'(\varphi\land \psi)=\bbot$)  then by  $\mathsf{r}_4$,  $\nabla (\varphi\land \psi)\notin T$.
By $\mathsf{r}_{13}$ and $\mathsf{r}_{1}$  we have that either $\nnot \nabla \varphi\in T$ or $\nnot \nabla \psi\in T$ (so either $v'(\varphi)=\bbot$ or $v'(\psi)=\bbot$) hence  $v'(\varphi\land \psi)=\bbot=v'(\varphi)\tilde{\land} v'(\psi)$.


\item $v'(\varphi\lor \psi)=v'(\varphi)  \lor  v'(\psi)$

If (i)  we have that $\nabla (\varphi\lor \psi),\nabla \nnot (\varphi\lor \psi)\in T$.
From $\nabla \nnot (\varphi\lor \psi)\in T$, by 
$\mathsf{r}_{14}$ and $\mathsf{r}_{15}$
we obtain that
 $\nabla \nnot \varphi,\nabla \nnot \psi \in T$ 
($v'(\varphi)\neq\ttop\neq v'(\psi)$). 
From $\nabla (\varphi\lor \psi)\in T$, by 
$\mathsf{r}_{16}$,
either $\nabla\varphi$ or $\nabla \psi$ are in $ T$ ($v'(\varphi)\neq\bbot$ or $v'(\psi)\neq\bbot$). 
Hence, if $\nabla\varphi,\nabla\psi\in T$ then $v'(\varphi\lor \psi)=v(\varphi\lor \psi)=v(\varphi)\tilde{\lor} v(\psi)=v'(\varphi)  \tilde{\lor}  v'(\varphi)$.
Otherwise, if $\nabla \varphi\notin T$, then by $\mathsf{r}_{1}$ we conclude that  $\nnot\nabla \varphi\in T$ and hence
$v'(\varphi\lor \psi)=v'(\psi)=\bbot\,\tilde{\lor}\, v'(\psi)=v'(\varphi)\,\tilde{\land}\, v'(\psi)$. The case $\nabla \psi\notin T$ is similar.

If (ii)  we have that  $\nnot\nabla \nnot (\varphi\lor \psi)\in T$ (so  $v'(\varphi\lor \psi)=\ttop$) then $ \nnot (\varphi\lor \psi)\notin T$ by $\mathsf{r}_{4}$.
From $\mathsf{r}_{19}$ and $\mathsf{r}_{1}$ either $\nnot \nabla \nnot \varphi\in T$ or $\nnot \nabla \nnot \psi\in T$.
(so either $v'(\varphi)=\ttop$ or $v'(\psi)=\ttop$) hence  $v'(\varphi\lor \psi)=\ttop=v'(\varphi)\tilde{\lor} v'(\psi)$.

{
If (iii)  we have that  $\nnot\nabla (\varphi\lor \psi)\in T$ (so $v'(\varphi\lor \psi)=\bbot$) and so $\nabla (\varphi\lor \psi)\notin T$  by $\mathsf{r}_4$.
By 
$\mathsf{r}_{17}$, $\mathsf{r}_{18}$ we have that $ \nabla \varphi, \nabla  \psi\notin T$,
and by $\mathsf{r}_1$,
$\nnot \nabla \varphi, \nnot\nabla  \psi\in T$ (so $v'(\varphi)=v'(\psi)=\bbot$). Thus,
 $v'(\varphi\lor \psi)=v'(\varphi)  \lor  v'(\psi)=\bbot\tilde{\lor} \bbot=\bbot$.}

\item $v'(\bot)=\bbot$ and  $v'(\top)=\ttop$

Directly from rules 
$\mathsf{r}_{20}$ and $\mathsf{r}_{21}$.
\qedhere
\end{enumerate} 

Regarding the preceding Proposition, note that, in order to show that $v'(\xi)$ is well defined for every $\xi\in Fm_{\Sigma^{\nabla}}$  (and that $v'\in \Val(\Mt^\nabla)$), we only need to consider rules in $\R_\nabla$ instantiated with formulas in $\sub(\xi)$. 
This observation will be crucial in 
the proof of Theorem~\ref{analytic2de3}.


\end{proof}


We now proceed to explain how a single-conclusion axiomatization (for a logic extending  $\ISL$)
can be extracted from the multiple-conclusion rules of Theorem~\ref{nablaax}.
We shall need a few technical lemmas. In the next one, 
$ \prod_i{\Mt_i}$  denotes the product of a family of matrices $\{ \Mt_i : i \in I \}$.


\begin{lemma}\label{lem:emb}
  Given a class $\{ \Mt_i : i \in I \}$ of $\Sigma$-matrices, 
  we have the following embeddings:
 $$
 \prod_i{\Mt_i}\hookrightarrow (\prod_i\Mt_i)^\nabla \hookrightarrow \prod_i({\Mt_i}^\nabla).
 $$
\end{lemma}
\begin{proof}
 The fist embedding is simply the identity function. 
The second one is also the identity for the elements in $\prod_i({\Mt_i})$, whereas  $\ttop$ is sent to $\prod_i\{\ttop\}$ and
$\bbot$ to $\prod_i\{\bbot\}$. 
\end{proof}

The following result is an immediate consequence of  Lemma~\ref{lem:emb}.
\begin{lemma}\label{lem:emblogs}
 For every class $\CMt$ of $\Sigma$-matrices, we have:
\begin{enumerate}[(i)]
  \item ${\derm_{\PR(\widehat{\CMt})}} \subseteq {\derm_{\widehat{\PR(\CMt)}}} \subseteq {\derm_{\PR(\CMt)}}$.
 \item ${\derm_{\PR({\CMt}^\nabla)}} \subseteq {\derm_{(\PR(\CMt))^\nabla}} \subseteq {\derm_{\PR(\CMt)}}$.
\end{enumerate}
\end{lemma}

Let $\R$ be a 
a set of single-conclusion rules. Recall that
$\R_\nabla$ is a set of multiple-conclusion rules, and
we 
abbreviate
 ${\ders_{\R\cup \R_\nabla}}={\ders_{\derm_{\R\cup \R_\nabla}}}$.

\begin{corollary}\label{cor:stomax}
 Let $\CMt$ be a class of $\Sigma$-matrices, and let 
$\R$ be a set of single-conclusion rules. 
If ${\ders_\R} = \Log \CMt =\Log {\widehat{\CMt}} $, 
 then ${\ders_{\R\cup \R_\nabla}} = \Log {\CMt^\nabla} $.
\end{corollary}
\begin{proof}
From $\Log \CMt =\Log {\widehat{\CMt}} =\ders_\R$ we have $\derm_\R\simeq \derm_{\PR(\widehat{\CMt})}\simeq \derm_{\PR(\CMt)}$ by Remark~\ref{rem:svsm} {(ii)}.
From Lemma~\ref{lem:emblogs} (i) and Remark~\ref{rem:svsm} we obtain that $\derm_\R\simeq\derm_{\widehat{\PR(\CMt)}}$.
By Theorem~\ref{nablaax} we conclude that $\derm_{(\PR(\CMt))^\nabla}=\derm_{\R\cup \R_\nabla}$.
Finally, from Lemma~\ref{lem:emblogs} (ii) and  Remark~\ref{rem:svsm} we have  
$\Log {\CMt^\nabla} = \Log \PR(\CMt)^\nabla $. 
Hence
 $\ders_{\CMt^\nabla}$ is axiomatized by $\R\cup \R_\nabla$. 
\end{proof}

Joining Theorem~\ref{disjax} and Corollary~\ref{cor:stomax}, we obtain the following recipe for capturing the effect of 
adding $\nabla$ to  single-conclusion axiomatizations.

\begin{corollary}
\label{cor:wck}
Let $\CMt$ be a class of $\Sigma$-matrices and let
$\R$ be a set of 
single-conclusion
rules.
If ${\ders_\R} = \Log \CMt =\Log {\widehat{\CMt}} $, then ${\ders_{({R\cup R_\nabla})^\lor}} = \Log {{\CMt}^\nabla} $.
\end{corollary}



\begin{example}
\label{ex:logis}
Let $\Mt_4 = \tuple{\alg{DM}_4,\uparrow\! a}$
be the four-element matrix that defines the Belnap-Dunn logic $\B$.
By Corollary~\ref{lem:newtopmat}, we know that 
$\B 
= \Log \Mt_4
= \Log \widehat{\Mt}_4
$.
%
%
Hence, from Corollaries~\ref{cor:newtopmat} and~\ref{cor:wck} we can obtain a Hilbert axiomatization for $\ISL=\Log \Mt_4^\nabla=\Log \tuple{\alg{IS}_6,\uparrow\! a}$.
Let $\R_\B$ be 
the Hilbert-style calculus used in~\cite{F97} to axiomatize  $\Log \B$
(expanded with the rules introduced in~\cite[p.~1065]{AlPrRi16}
to account for the constants).
\begin{center}
 
  \begin{tabular}{ccc}
  \\
    \AxiomC{$p \land q$}
    \UnaryInfC{$p$}
    \DisplayProof &
    \AxiomC{$p \land q$}
    \UnaryInfC{$q$}
    \DisplayProof &
    \AxiomC{$p$}
    \AxiomC{$q$}
    \BinaryInfC{$p \land q$}
    \DisplayProof \\ \\ 

    \AxiomC{$p$}
    \UnaryInfC{$p \lor q$}
    \DisplayProof &
    \AxiomC{$p \lor q$}
    \UnaryInfC{$q \lor p$}
    \DisplayProof &
    \AxiomC{$p \lor p$}
    \UnaryInfC{$p$}
    \DisplayProof \\ \\ 
 
    \AxiomC{$p \lor (q \lor r)$}
    \UnaryInfC{$(p \lor q) \lor r$}
    \DisplayProof &
    \AxiomC{$p \lor (q \land r)$}
    \UnaryInfC{$(p \lor q) \land (p \lor r)$}
    \DisplayProof &
    \AxiomC{$(p \lor q) \land (p \lor r)$}
    \UnaryInfC{$p \lor (q \land r)$}
    \DisplayProof \\ \\ 
 
    \AxiomC{$p \lor r$}
    \UnaryInfC{$\neg \neg p \lor r$}
    \DisplayProof &
    \AxiomC{$\neg \neg p \lor r$}
    \UnaryInfC{$p \lor r$}
    \DisplayProof &
    \AxiomC{$\neg (p \lor q) \lor r$}
    \UnaryInfC{$ (\neg p \land \neg q) \lor r$}
    \DisplayProof \\ \\ 
 
    \AxiomC{$(\neg p \land \neg q) \lor r$}
    \UnaryInfC{$\neg (p \lor q) \lor r$}
    \DisplayProof &
    \AxiomC{$\neg (p \land q) \lor r$}
    \UnaryInfC{$(\neg p \lor \neg q) \lor r$}
    \DisplayProof &
    \AxiomC{$(\neg p \lor \neg q) \lor r$}
    \UnaryInfC{$\neg (p \land q) \lor r$}
    \DisplayProof 
\\ \\ 

  \end{tabular} 
  $$\frac{}{\top}\qquad\frac{}{\nnot\bot}\qquad \frac{\bot\lor p}{p}\qquad \frac{\nnot\top\lor p}{p}$$

  \end{center}


Then  $\ISL$ is axiomatized by $(\R_\B\cup \R_\nabla)^\lor$, which is the result of adding
to $\R_\B$ the following rules:

 $$  \frac{}{\nabla p\lor \nnot \nabla p}  {\ \mathsf{r}^\lor_1} \qquad
   { \frac{\nabla p \lor r  }{\nnot\nabla\nnot\nabla  p\lor r}}  {\ \mathsf{r}^\lor_2} \qquad
    { \frac{ \nabla\nabla  p\lor r}{ \nabla p \lor r }}  {\ \mathsf{r}^\lor_3} \qquad
   \frac{\nabla p\lor r\,,\, \nnot\nabla p\lor r}{r}  {\ \mathsf{r}^\lor_4}
 $$  
  
  $$
    { 
   \frac{\nnot\nabla p\lor r}{\nabla \nnot p\lor r}} {\ \mathsf{r}^\lor_5}\qquad
   \frac{\nabla \nnot\nnot p\lor r }{\nabla p\lor r} {\ \mathsf{r}^\lor_6} \qquad 
      \frac{\nabla p\lor r}{\nabla \nnot\nnot p \lor r} {\ \mathsf{r}^\lor_7}$$
 
$$\frac{\nabla(p\land q)\lor r}{\nabla p\lor r} {\ \mathsf{r}^\lor_8}\qquad 
 \frac{\nabla(p\land q)\lor r}{\nabla q\lor r} {\ \mathsf{r}^\lor_9} \qquad
 \frac{\nabla \nnot (p\land q)\lor r}{\nabla \nnot p\lor \nabla \nnot q \lor r} {\ \mathsf{r}^\lor_{10}}\qquad 
$$

 $$ 
  {
  \frac{ \nabla \nnot p \lor r}{ \nabla \nnot (p\land q) \lor r}} {\ \mathsf{r}^\lor_{11}} \qquad 
  {
  \frac{\nabla \nnot q\lor r}{\nabla \nnot (p\land q)\lor r }}{\ \mathsf{r}^\lor_{12}} \qquad 
  {
  \frac{ \nabla p\lor r \,,\,  \nabla q\lor r}{ \nabla  (p\land q)\lor r} }{\ \mathsf{r}^\lor_{13}}$$




 $$
 \frac{\nabla \nnot (p\lor  q)\lor r}{\nabla\nnot p\lor r}{\ \mathsf{r}^\lor_{14}}\qquad 
 \frac{\nabla \nnot (p\lor q)\lor r}{\nabla\nnot q\lor r}{\ \mathsf{r}^\lor_{15}} 
 \qquad \frac{\nabla(p\lor  q)\lor r}{ \nabla  p\lor\nabla  q \lor r}{\ \mathsf{r}^\lor_{16}} 
$$ 


$$
   \frac{ \nabla   p\lor r}{  \nabla(p\lor q) \lor r}  
  {\ \mathsf{r}^\lor_{17}} 
  \qquad
   \frac{  \nabla   q\lor r}{  \nabla(p\lor q) \lor r}  
  {\ \mathsf{r}^\lor_{18}} 
  \qquad
  \frac{ \nabla \nnot  p\lor r \,,\,   \nabla \nnot  q\lor r}{  \nabla \nnot (p\lor q) \lor r}  
  {\ \mathsf{r}^\lor_{19}} 
$$

$$\frac{}{\nnot \nabla \bot}
\ \mathsf{r}^\lor_{20} 
\qquad\frac{}{\nnot \nabla \nnot \top}
\
 \mathsf{r}^\lor_{21} 
  $$

\end{example}

In the next Subsection we are going to  apply Corollary~\ref{cor:wck} to axiomatize (relatively to 
$ 
\B^\nabla
$), 
some extensions of $\ISL$ that are characterized by 
$\Mt^\nabla$ for some matrix $\Mt$ that is a model of 
$\B$. 

\subsection{Adding $\nabla$ to super-Belnap logics}
\label{ss:addsb}

As observed earlier, 
$\ISL = {\ders_\CMt}$, where $ \CMt$ is the following class of matrices:
$$
\CMt :=
\{ \la \A, F \ra :  \A \in \ISA, F \subseteq A \text{ is a (non-empty) lattice filter} \}.
$$
Thus, each subclass  $\CMt' \subseteq \CMt$  (it suffices to consider
those $\CMt'$ consisting of reduced matrices) defines a logic
$\ders_{\CMt'}$ which is an
 extension of  $\ISL$. 
 We have seen with
Corollary~\ref{cor:cardlat} that there are at least
continuum many of these, and Corollary~\ref{cor:latlog} suggests
that the structure of the lattice of extensions of $\ISL$ is quite complex
(see~\cite{AlPrRi16,Adam} for analogous considerations on the lattice of super-Belnap logics). 
A systematic study of this lattice lies outside the scope of the present paper
and even beyond our present grasp on IS-logics;
however, in this Subsection we consider a few extensions of $\ISL$ that are defined by substructures 
of $\la \Al[IS]_6, \uparrow \!  a \ra $, illustrating how our methods
can be used to axiomatize them.

The following result, which is an immediate consequence of Corollary~\ref{cor:wck},
shows that Example~\ref{ex:logis} smoothly generalizes to all super-Belnap logics.


\begin{proposition}\label{prop:axnablasuperbelnap}
Let $\CMt$ be a class of models of $\B$. If $\Log \CMt $ is axiomatized relative to  $\B$ by a set of single conclusion rules $\R$, then $\Log \CMt^\nabla $ is  also axiomatized by $\R$ relative to $\B^\nabla$.
\end{proposition}
%
Let $\CMt_1$ and $\CMt_2$ be classes of 
models of  $\B$
such that $ \Log {\CMt_1} = \Log {\CMt_2} $.
Then  $\Log {\CMt_1} $ and $\Log \CMt_2 $ are axiomatized by the same set $\R$ of single-conclusion rules.
Hence, $\Log {\CMt^\nabla_1} =\Log {\CMt^\nabla_2} $ is 
axiomatized by the set $\R_\nabla^\lor$ defined above.
This entails, in particular, that,
if a super-Belnap logic $\Lo$ is finitary (resp.~finitely axiomatized), then
$\Lo^{\nabla}$
(defined as in Corollary~\ref{cor:latlog})
is also finitary (resp.~finitely axiomatized).
Since the lattice of super-Belnap logics
contains continuum many finitary logics~\cite[Cor.~8.17]{Adam},
the above considerations allow us to obtain the following
sharpening of Corollary~\ref{cor:cardlat}.

%

 \begin{proposition}
\label{prop:cardlatfin}
There are (at least) continuum many finitary extensions of  $\ISL$.
\end{proposition}

The super-Belnap logics considered below are the so-called \emph{Exactly True Logic} $\mathcal{ETL}$ of~\cite{PiRi}
(which is the 1-assertional logic of the variety of De Morgan algebras),
G.~Priest's \emph{Logic of Paradox} $\mathcal{LP}$,  the two logics $\mathcal{K_{\leq}}$ and $\mathcal{K}_1$
named after S.~C.~Kleene, and
classical logic  $\mathcal{CL}$.
$\mathcal{K_{\leq}}$ is the order-preserving logic of the variety of Kleene algebras,
and $\mathcal{K}_1$ is the 1-assertional logic associated to the same variety (see~\cite{AlPrRi16} for further details). Proposition~\ref{prop:sb} below shows that
each of these
logics can be axiomatized, relative to $\B$, by a combination of the following rules:
\begin{gather}
  p \land( \nnot p \lor q) \vdash q
 \tag{DS} \label{Eqn:DS} \\
  (p \land \nnot p) \lor q \vdash q
 \tag{$K_1$} \label{Eqn:K1} \\
(p \land \nnot p) \lor r \vdash q \lor \nnot q \lor r 
 \tag{$K_{\leq}$} \label{Eqn:Kleq} \\
\emptyset \vdash p \lor \nnot p 
 \tag{EM} \label{Eqn:EM}
 \end{gather}
(Regarding the names of the above rules, the $K$'s are suggestive of Kleene's logics,
\eqref{Eqn:DS} stands for \emph{Disjunctive Syllogism}
and \eqref{Eqn:EM} for \emph{Excluded Middle}.)

\begin{proposition}[\cite{AlPrRi16},~Thm.~3.4]
\label{prop:sb}
$ 
$
\begin{enumerate}[(i)]

 \item $\mathcal{ETL}  = \Log\tuple{\alg{DM}_4, \{1 \}} =\B +$\eqref{Eqn:DS}. 

\item  $\mathcal{LP} = \Log \tuple{\alg{K}_3,\uparrow \!  a} = \B +$\eqref{Eqn:EM}. 

 \item 
 $\mathcal{K}_1 = \Log\tuple{\alg{K}_3, \{ 1 \} } =   \B +$\eqref{Eqn:K1}.  

\item $\mathcal{K}_{\leq} = \Log \{\tuple{\alg{K}_3,\uparrow \!  a}, \tuple{\alg{K}_3, \{1 \}  } \} = \B + $\eqref{Eqn:Kleq}. 

\item $\mathcal{CL} = \Log  \tuple{\alg{B}_2, \{1 \}  }  = \B + $\eqref{Eqn:DS}$+$\eqref{Eqn:EM}.

\end{enumerate}
 \end{proposition}


\begin{theorem}
 For logics above $\ISL$
 we have the following relative axiomatizations:
 \begin{enumerate}[(i)]
 \item $\Log \tuple{\alg{IS}_6,\uparrow \!  1} =\mathcal{ETL}^{\nabla} = \ISL+ $\eqref{Eqn:DS}. 
  \item $\Log \tuple{\alg{IS}_5,\uparrow \!  a} =\mathcal{LP}^{\nabla}  = \ISL+$\eqref{Eqn:EM}. 

 \item $\Log \tuple{\alg{IS}_5,\uparrow \!  1} = \mathcal{K}^{\nabla}_1= \ISL+$\eqref{Eqn:K1}. 
  \item $\Log \{\tuple{\alg{IS}_5,\uparrow \!  a},\tuple{\alg{IS}_5,\uparrow \!  1} \}=\mathcal{K}^{\nabla}_{\leq} = \ISL+$\eqref{Eqn:Kleq}.  
  
 \item $\Log \tuple{\alg{IS}_4,\uparrow \!  1} =  \mathcal{CL}^{\nabla}= \ISL + $\eqref{Eqn:DS}$+$\eqref{Eqn:EM}.
\end{enumerate}
\end{theorem}
 \begin{proof}
 The statement follows directly from Proposition~\ref{prop:axnablasuperbelnap} and Proposition~\ref{prop:sb},  having noticed that,
for $x\in \{1,a\}$, we have
 $\tuple{\alg{DM}_4,\uparrow \!  x}^\nabla= \tuple{\alg{IS}_6,\uparrow \!  x}$,
 $\tuple{\alg{K}_3,\uparrow \!  x}^\nabla = \tuple{\alg{IS}_5,\uparrow \!  x}$ and $\tuple{\alg{B}_2,\uparrow \!  1}^\nabla= \tuple{\alg{IS}_4,\uparrow \!  1}$.
%
%
%
%
%
%
%
%
%
%
\end{proof}

\subsection{Other extensions of $\ISL$}
\label{ss:oth}


In this Subsection we consider a few examples of 
extensions of $\ISL$ (defined by substructures of the matrix $\tuple{\alg{IS}_6,\uparrow \!  a}$)



%

Given a $\Sigma$-matrix 
$\Mt=\tuple{\A,D 
}$, 
a set of axioms $\Ax\subseteq Fm_{\Sigma}$ and a set of rules $\R\subseteq \wp(Fm_{\Sigma})\times Fm_{\Sigma}$, 
we write
 $\Val_\Mt^\Ax$ for the set of valuations on $\Mt$ such that $v(\varphi^\sigma)\subseteq D$ for every $\varphi\in \Ax$ substitution $\sigma$, 
 and 
 $\Val_\Mt^\R$ for the set of valuations on $\Mt$ such that $v(\Gamma^\sigma)\subseteq D$ implies $\varphi^\sigma\in D$ for every $\frac{\Gamma}{\varphi}\in \R$ and substitution $\sigma$.

The following result (whose simple proof we omit) is a corollary of~\cite[Lemma~2.7]{soco} 
and will be very useful to show relative axiomatization results 
(
this technique is used in~\cite{CaMaAXS} to obtain general modular semantics for axiomatic extensions of a given logic).
In item (ii), $\Mt^{\omega}$ is a shorthand for $\prod_{i<\omega}\Mt$.

\begin{proposition}\label{proofbyvals}
Let $\Mt=\tuple{\A,D 
}$ be a $\Sigma$-matrix.
Given $\Ax\subseteq Fm_{\Sigma}$ and $\R\subseteq \wp(Fm_{\Sigma})\times Fm_{\Sigma}$.
We have that: 
\begin{enumerate}[(i)]
 \item
 $\Val_\Mt^\Ax$ is a  complete semantics 
for 
$\Log(\Mt) + \Ax$. 
 \item
 $\Val_{\Mt^\omega}^\R$ is  a complete semantics  for
$\Log(\Mt) + \R$.  
\end{enumerate}
 \end{proposition}

We are now ready to give an axiomatization relative to $\ISL$
of ${\vdash^{1}_{\ISA}}$,
the 1-assertional logic of the class $\ISA$ (i.e.~three-valued 
\L ukasiewicz(-Moisil) logic; cf.~Proposition~\ref{prop:quas}).
This is the first item of 
Theorem~\ref{th:axnotinshape} below.
The logic axiomatized by the second item 
is the order-preserving logic of the variety 
$\VV (\Al[IS]_3)$
of three-valued  \L ukasiewicz(-Moisil) algebras,
which is also $\Log\{ \la \alg{IS}_3, \{ \ttop \} \ra, \la \alg{IS}_3, \uparrow \!  0 \ra \}$.

\begin{theorem}\label{th:axnotinshape}
$
$
 \begin{enumerate}[(i)]
\item $\Log \la \alg{IS}_6, \{ \ttop \} \ra = \Log \la \alg{IS}_5, \{ \ttop \} \ra = \Log \la \alg{IS}_4, \{ \ttop \} \ra  =
 \Log \la \alg{IS}_3, \{ \ttop \} \ra=  {\vdash^{1}_{\ISA}} =  \ISL + p \vdash \nnot \nabla \nnot p $. 

\item  $\Log\{ \la \alg{IS}_3, \{ \ttop \} \ra, \la \alg{IS}_3, \uparrow \!  0 \ra \} = {\vdash^{\leq}_{\VV (\Al[IS]_3)}} = \ISL +$\eqref{Eqn:Kleq}
$ + 
\nnot p\lor r, \nabla p\lor r\vdash p\lor r 
+ \nnot \nabla p\lor r \vdash \nnot p\lor r$.

\item $\Log(\la \alg{IS}_6, \uparrow \!  0 \ra) = \Log(\la \alg{IS}_5, \uparrow \!  0 \ra) = \Log(\la \alg{IS}_4, \uparrow \!  0 \ra) =
  \Log(\la \alg{IS}_3, \uparrow \!  0 \ra) = \ISL + p \lor \nnot \nabla p$.
\end{enumerate}

\end{theorem}
\begin{proof}
 (i). For the equalities 
$
\Log(\la \alg{IS}_6, \{ \ttop \} \ra) 
= \Log(\la \alg{IS}_5, \{ \ttop \} \ra) 
= \Log(\la \alg{IS}_4, \{ \ttop \} \ra )
= \Log(\la \alg{IS}_3, \{ \ttop \} \ra)
$,
it suffices to observe that
$
\la \alg{IS}_6, \{ \ttop \} \ra^*  
=  \la \alg{IS}_5, \{ \ttop \} \ra^* 
=  \la \alg{IS}_4, \{ \ttop \} \ra^* 
=  \la \alg{IS}_3, \{ \ttop \} \ra 
$.
%
Consider $\Mt=\tuple{\alg{IS}_6, \uparrow \!  b}^\omega$ and $\R=\{p \vdash \nnot \nabla \nnot p\}$.
Note for every $v\in \Val_{\Mt^\omega}$ we have that $v(\nnot\nabla\nnot p)\in D$ iff $v(p)=\{\ttop\}^\omega$.
Hence, from $v\in \Val_{\Mt^\omega}^\R$ we have that $v(A)\in D^\omega$ iff $v(A)=\{1\}^\omega$.
As $\R$ is sound in $\tuple{\alg{IS}_6, 1}^\omega$ we obtain the equality $\Val_{\Mt^\omega}^\R=\Val_{\tuple{\alg{IS}_6, 1}^\omega}$,
thus $\ISL+p \vdash \nnot \nabla \nnot p=\Log(\tuple{\alg{IS}_6, 1}^\omega)=\Log(\tuple{\alg{IS}_6, 1})$.


(ii). That 
$\Log(\{ \la \alg{IS}_3, \{ \ttop \} \ra, \la \alg{IS}_3, \uparrow \!  0 \ra \}) $
is the order-preserving logic of the variety $\VV (\alg{IS}_3) $
follows from the  observation that ${\vdash^{\leq}_{\VV (\K)}} = {\vdash^{\leq}_{\K}} $ holds for any class $\K$.
%
Applying Theorem~\ref{disjax} to 
the multiple-conclusion axiomatization 
for $\Log(\{ \la \alg{IS}_3, \{ \ttop \} \ra, \la \alg{IS}_3, \uparrow \!  0 \ra \})$ presented in Example~\ref{ex:2with3} in the following Section, we conclude that
collecting all the $\mathsf{s}_i^\lor$ for $1\leq i \leq 21$ provides a single-conclusion axiomatization of $\Log(\{ \la \alg{IS}_3, \{ \ttop \} \ra, \la \alg{IS}_3, \uparrow \!  0 \ra \})$.
The latter equality in the statement follows as the new rules, corresponding to $\mathsf{s}^\lor_{15}$, $\mathsf{s}^\lor_{16}$ and $\mathsf{s}^\lor_{20}$,  are exactly ones that fail in
  $\tuple{\alg{IS}_6,\uparrow\! a}$.
  %

(iii). For the equalities 
$\Log(\la \alg{IS}_6, \uparrow \!  0 \ra) = \Log(\la \alg{IS}_5, \uparrow \!  0 \ra) = \Log(\la \alg{IS}_4, \uparrow \!  0 \ra) =
  \Log(\la \alg{IS}_3, \uparrow \!  0 \ra )$,
  it suffices to observe that
$
 \la \alg{IS}_6, \uparrow \!  0 \ra^* 
=  \la \alg{IS}_5, \uparrow \!  0 \ra^* 
=  \la \alg{IS}_4, \uparrow \!  0 \ra^* =
   \la \alg{IS}_3, \uparrow \!  0 \ra
$.
Let $\R=\{p \lor \nnot \nabla p\}$.
In order to obtain the last equality, in the light of Proposition~\ref{proofbyvals}~{(i)}, it is enough to show that $\Val_{\tuple{\alg{IS}_3, \uparrow 0}}=\Val^\R_{\tuple{ \alg{IS}_6, \uparrow a}}$.
This follows from the fact that given $v\in \Val^\R_{\tuple{ \alg{IS}_6, \uparrow  a}}$ we must have that for every formula $\phi$ both $v(\phi)\neq b$ and $v(\phi) \neq 0$ (therefore also $v(\phi) \neq 1$), and the fact that $\R$ is sound w.r.t.~$\tuple{\alg{IS}_3, \uparrow \!  0}$.
\end{proof}

\subsection{Analtytic calculi}
\label{ss:anal}

Let $\Lambda\subseteq Fm$
and let $\R$ be a set of multiple-conclusion rules.
We write\footnote{
Note that 
in general 
$\derm^\Lambda_\R$ is not a multiple-conclusion consequence relation. It still satisfies dilution and cut for set properties, but only weaker versions of overlap and substitution invariance.}
 $\Gamma\derm^\Lambda_\R\Delta$ when there exists an $\R$-proof of $\Delta$ from $\Gamma$ 
 where only formulas in~$\Lambda$ occur. 
Let $\Phi\subseteq Fm$. We say that $\R$ is \emph{$\Phi$-analytic} if 
when $\Gamma\derm_\R\Delta$ then 
$\Gamma\derm^{\Upsilon_\Phi}_\R\Delta$
with $\Upsilon=\sub(\Gamma\cup \Delta)$ and $\Upsilon_\Phi=\Upsilon\cup\{A^\sigma : A\in \Phi, \sigma:P\to \Upsilon\}$.
Intuitively, this means that an $\R$-proof of $\Delta$ from $\Gamma$ needs only to use formulas which are subformulas of $\Gamma\cup \Delta$, or instances of $\Phi$ with such subformulas. Hence, formulas in $\Upsilon_\Phi$ can be seen as 
`generalized subformulas'.  

Given $\Phi\subseteq Fm$, let $\Phi^\nabla: =\Phi\cup\{\nabla p, \nnot \nabla p, \nabla \nnot p, \nnot \nabla \nnot p \}$.
The Theorem below is a refinement of Theorem~\ref{nablaax} that applies 
when we depart from calculus that is  analytic, entailing that 
the operation described in Theorem~\ref{nablaax}
preserves analyticity.

\begin{theorem}\label{analytic2de3}
 Let  $\CMt$ be a class of $\Sigma$-matrices.
If $\R$ is an $\Phi$-analytic axiomatization of $\derm_{\widehat{\CMt}}$  
 then $\R\cup \R_\nabla$ is an $\Phi^\nabla$-analytic axiomatization of
 $\derm_{\CMt^\nabla}$.
\end{theorem}

\begin{proof}
The proof can be easily obtained by adapting the proof of Theorem~\ref{nablaax}.
Let 
 $\Upsilon=\sub(\Gamma\cup \Delta)$ and  
 $\Lambda=
\Upsilon_{\Phi^\nabla}$.
 Assume  $\Gamma\not\derm_{\R\cup \R_\nabla}^{\Lambda}\Delta$. 
 Then, by cut for sets, there is a partition $\tuple{T,F}$ of $\Lambda$ such that $\Gamma\subset T$, $\Delta\subset F$ and $T\not\derm_{\R\cup \R_\nabla}^{\Lambda}F$.
  Since $\R\subseteq \R\cup \R_\nabla$, we know that  $T\not\derm_{\R\cup \R_\nabla}^{\Lambda}F$. Therefore, since 
  $\Upsilon_{\Phi}\subseteq \Upsilon_{\Phi^\nabla}$, 
 by $\Phi$-analyticity 
  of $\R$ we have that 
 %
 %
 $T\not\derm_{\widehat{\CMt}}F$ and we can
 pick $v\in \mathsf{Hom}_{\Sigma}(L_{\Sigma^\nabla},\widehat{\Mt})$
 for some $\Mt\in \CMt$ such that $v(T)\subseteq D$ and $v(F)\cap D=\emptyset$. 
Noting that for every $A\in \Upsilon$ we have $\nabla A, \nabla \nnot A, \nnot \nabla A, \nnot \nabla \nnot A\in \Upsilon_{\mathcal{S}_\nabla}=\Lambda$, we can define 
$v':\Upsilon\to \Mt^\nabla$ as in Theorem~\ref{nablaax}. 
That $v'$ respects all the connectives (and is therefore a partial $\Mt^\nabla$ valuation)
follows from the fact that in the proof of Theorem~\ref{nablaax} we only used instances of the rules using formulas in $\Upsilon$ yielding formulas in $\Upsilon_{\mathcal{S}_\nabla}=\Lambda$. 
As $\Mt$ is a matrix, $v'$ can be extended to a total valuation and therefore $\Gamma\not\derm_{\CMt^\nabla}\Delta$, thus concluding the proof.
\end{proof}

The papers \cite{CaMa19,CaMaXX} 
 introduced a general method for obtaining analytic calculi for logics given by (partial non-deterministic) matrices 
whenever a certain expressiveness requirement is met. 
In particular,
for the logic determined by a single 
matrix
$\Mt = \la \A, D \ra$, it suffices that  
$\Mt$ be \emph{monadic}~\cite[p.~265]{ShSm78}. This means that,
for all $x,y \in A$ with $x\neq y$, there is a one-variable separating formula, that is,  a formula $\phi(p)$ such that $\phi(x)\in D$ and $\phi(y)\notin D$ (or vice versa). 

From now on, let us fix the separating set 
$\mathcal{S}: =\{p,\nnot p\}$.
Applying the above-described method, we obtain the following axiomatization for
$\B$.



\begin{example}
The matrix $\tuple{\alg{DM}_4,\uparrow\! a}$ is monadic with set of separators $\mathcal{S}$.
We can therefore apply the method introduced in~\cite{CaMaXX} to we obtain the following
$\mathcal{S}$-analtytic axiomatization for $\B=\Log \tuple{\alg{DM}_4,\uparrow\! a} $:

 $$
 \frac{p}{\nnot\nnot p}\qquad
  \frac{\nnot\nnot p}{p}
 $$

$$
\frac{p\land q}{p}\quad
\frac{p\land q}{q}\quad
 \frac{p,q}{p\land q}\qquad
 \frac{\nnot p}{\nnot(p \land q)}\quad
\frac{\nnot q}{\nnot(p \land q)}\quad \frac{\nnot(p \land q)}{\nnot p, \nnot q}
$$

 $$
 \frac{p}{ p \lor q}\qquad
  \frac{q}{ p \lor q}\qquad
  \frac{p\lor q}{p\,,\,q}\quad
   \frac{\nnot p\,,\,\nnot q}{\nnot (p\lor q)}
   \quad
       \frac{\nnot (p\vee q)}{\nnot p}
 \quad
    \frac{\nnot (p\vee q)}{\nnot q}
 $$
 
 $$\frac{}{\top}\qquad \frac{\nnot\top}{}\qquad\frac{}{\nnot\bot}\qquad \frac{\bot}{}$$
 

 Note that this axiomatization coincides with the one presented in \cite[Section 9]{Adam}.
 Theorem~\ref{analytic2de3} then tells us that we can obtain an $\mathcal{S}_\nabla$-analtytic axiomatization of $\ISL$ by adding $\R_\nabla$ to the above rules.
\end{example}

%




\begin{example}\label{ex:2with3}

In~\cite[Example 5]{CaMa19} we showed 
that the following rules provide an $\mathcal{S}$-analtytic axiomatization of Kleene's logic of order
$\mathcal{K}_{\leq} = \Log \{ \la \alg{K}_3, \{ 1\} \ra, \la \alg{K}_3, \uparrow \!  a \ra \}  $.

$$\frac{p\,,\, q}{\;p\land q\;}\ _{\mathsf{s}_1}\quad 
 \frac{\;p\land q\;}{p}\ _{\mathsf{s}_2}\quad \frac{\;p\land q\;}{q}\ _{\mathsf{s}_3} 
 \quad \frac{\nnot p}{\;\nnot(p\land q)\;}\ _{\mathsf{s}_4}  \quad \frac{\nnot q}{\;\nnot(p\land q)\;}\ _{\mathsf{s}_5}
 \quad  \frac{\;\nnot(p\land q)\;}{\;\nnot p\,,\, \nnot q\;}\ _{\mathsf{s}_6}$$

   $$\frac{p}{\;p\lor q\;}\ _{\mathsf{s}_7}\quad \frac{q}{\;p\lor q\;}\ _{\mathsf{s}_8}\quad 
   \frac{\;\nnot(p\lor q)\;}{\nnot p}\ _{\mathsf{s}_9}
   \quad \frac{\;\nnot(p\lor q)\;}{\nnot q}\ _{\mathsf{s}_{10}}
  \quad   
   \frac{\nnot p\,,\,\nnot q}{\;\nnot(p\lor q)\;}\ _{\mathsf{s}_{11}} \quad
   \frac{\;p\lor q\;}{\;p\,,\, q\;}\ _{\mathsf{s}_{12}}
  $$ 
  
  $$\quad\frac{p}{\;\nnot \nnot p\;}\, _{\mathsf{s}_{13}}\quad \frac{\;\nnot \nnot p\;}{p}\ _{\mathsf{s}_{14}}\quad {
  \frac{\;p\,,\, \nnot p\;}{q\,,\,\nnot q}\ _{\mathsf{s}_{15}}} $$

Since $\{ \la \alg{IS}_5,  \uparrow \!  1 \ra, \la \alg{IS}_5, \uparrow \!  a \ra \}=\{ \la \alg{K}_3, \{ 1\} \ra, \la \alg{K}_3, \uparrow \!  a \ra \}^\nabla$, 
by Theorem~\ref{analytic2de3}, we have that  adding  $\R_\nabla$ to the above rules gives us a $\mathcal{S}_\nabla$-analtytic axiomatization of
$\Log \{ \la \alg{IS}_5,\uparrow \!  1 \ra, \la \alg{IS}_5, \uparrow \!  a \ra \} $. 

The method of~\cite{CaMaXX} can also be applied directly
to obtain an  $\mathcal{S}$-analtytic axiomatization of  
$ {\vdash^{\leq}_{\VV (\Al[IS]_3)}}  = \Log \{ \la \alg{IS}_3, \{ \ttop \} \ra, \la \alg{IS}_3, \uparrow \!  0 \ra \} $.
Indeed, the $\nabla$-free fragment of $\Log \{ \la \alg{IS}_3, \{ \ttop \} \ra, \la \alg{IS}_3, \uparrow \!  a \ra \}$ 
is 
$\Log \{ \la \alg{K}_3, \{ 1 \} \ra, \la \alg{K}_3, \uparrow \!  a \ra \} $. The latter set of matrices can be viewed as a partial matrix~\cite[Section 2.2]{CaMaXX} 
and is monadic with separating set $\mathcal{S}$ (in which $\nabla$ does not occur). Therefore, the modularity of the method of~\cite{CaMaXX} tells us we just have to add the rules corresponding to $\nabla$. 
 Hence, it suffices to add  the rules:

  $${
  \frac{\nnot p\,,\,\nabla p}{p}\ _{\mathsf{s}_{16}}}\qquad
  {
  \frac{\nnot p}{p\,,\,\nnot\nabla p} \ _{\mathsf{s}_{17}}}\qquad \frac{p}{\nabla p}\ _{\mathsf{s}_{18}} \qquad
    \frac{p\,,\,\nnot\nabla p}{} \ _{\mathsf{s}_{19}}\qquad \frac{\nnot \nabla p}{\nnot p}\ _{\mathsf{s}_{20}}\qquad \frac{}{
    \nnot p, \nabla p}\ _{\mathsf{s}_{21}}$$
which give us 
the single-conclusion axiomatization relative to $\ISL$ mentioned in Theorem~\ref{th:axnotinshape} (ii).
\end{example}

\section{Conclusions and future work}
\label{sec:conc}

As we have shown,  the lattice of super-Belnap logics is embeddable in the lattice of extensions of $\ISL$.
This connection provides significant insight, but it also 
suggests
that
 fully describing the latter is at least as complex as describing the former, 
whose structure is still not completely understood (see~\cite{AlPrRi16}).
Obviously,   in the present study
 we have  only scratched the surface  of the general problem.
 A reasonable starting point for a systematic account of the extensions of $\ISL$ is to adapt 
 the various results and strategies in~\cite{Ri12b,AlPrRi16,Adam} to the richer setting of involutive Stone algebras.
We mention, in particular, the issues of 
characterizing the reduced models of extensions of $\ISL$,
and that of 
providing a general semantical description 
of the \emph{explosive extensions} (in the sense of~\cite{AlPrRi16,Adam}) of logics over $\ISL$.


An altogether different perspective on extensions of $\ISL$, which has not been considered in the present paper,
comes from the  observation made in~\cite[Sec.~6]{CaFi18} that $\ISL$ may be viewed as a 
paraconsistent logic, more precisely as a 
\emph{Logic of Formal Inconsistency} (LFIs) in the sense of N.~da Costa~\cite{dCo63}.
Indeed, $\ISL$ (and so its extensions) can be equivalently presented in a language that replaces the $\nabla$ connective
with either the consistency operator $(\circ)$ or the inconsistency operator $(\bullet)$
that are usually considered in the literature on LFIs.
One possible definition is $ \nabla \phi : = \nnot \circ \phi \lor \phi$, and, conversely,
one may define 
$\circ \phi : = \nnot \nabla ( \phi \land \nnot  \phi) $ and
$\bullet \phi : = \nabla (\phi \land \nnot \phi) $.

From a philosophical logic point of view, the advantage of the latter presentation
is that the operators $\circ$ and $\bullet$ have a clearer logical interpretation
than $\nabla$, namely, $\circ \phi$ means `$\phi$ is consistent' and 
$\bullet \phi$ means `$\phi$ is inconsistent'; 
%
on the other hand, $\nabla$ 
behaves very well from the points of view of algebraic logic and duality theory, 
for it satisfies the usual axioms for modal operators.
A more interesting observation is that, in the setting of LFIs,
 the (in)consistency operators are 
usually required to satisfy much weaker axioms
than those that result
 from the definitions $\circ \phi : = \nnot \nabla ( \phi \land \nnot  \phi) $ and
$\bullet \phi : = \nabla (\phi \land \nnot \phi) $ within  $\ISL$. 
This suggests a potentially fruitful  project for future research: namely,
a systematic study of more general algebraic structures 
(e.g.~De Morgan algebras endowed with a consistency operator)
corresponding to weaker logics (viewed as LFIs) that approximate 
$\ISL$ from below. 

\section*{Compliance with ethical standards}

Funding: Research funded by FCT/MCTES through national funds and when applicable co-funded by EU under the project UIDB/EEA/50008/2020
and by the Conselho Nacional de Desenvolvimento Científico e Tecnol\'ogico (CNPq, Brazil), under the grant 313643/2017-2 (Bolsas de Produtividade em Pesquisa - PQ).

Conflict of Interest:  The authors declare that they have no conflict of interest.

Ethical approval: This article does not contain any studies with human participants or animals.



\end{document}